\documentclass[11pt,reqno]{article} 
\usepackage[latin1]{inputenc}   % pour les accents
\usepackage{amsfonts,amsmath,layout}
\usepackage{mathrsfs}  
\usepackage{amssymb,mathabx,amsfonts,amsthm,amscd,stmaryrd,dsfont,esint,upgreek,constants,todonotes}
\DeclareFontFamily{OT1}{pzc}{}
\DeclareFontShape{OT1}{pzc}{m}{it}{<-> s * [1.10] pzcmi7t}{}
\DeclareMathAlphabet{\mathpzc}{OT1}{pzc}{m}{it}

\usepackage{color} % temporary

\definecolor{Red}{cmyk}{0,1,1,0.2}
\newcommand{\N}{\mathbb N}

\newcommand{\Z}{\mathbb Z}

\newcommand{\R}{\mathbb R}

\newcommand{\T}{\mathbb{T}}

\def\R{\mathbb R}
\def\N{\mathbb N}
\def\Z{\mathbb Z}
\def\E{\mathbb E}

\def\ep{\epsilon}

\def\dk{{\bf d}_1}

\newcommand{\be}{\begin{equation}}
\newcommand{\ee}{\end{equation}}
\def\1{{\bf 1}}
\def\inte{\int_{\T^d}}
\def\dive{{\rm div}}

\def\Pw{{\mathcal P}(\T^d)}
\def\Pk{{\mathcal P}(\T^d)}
\def\bx{{\boldsymbol x}}
\def\by{{\boldsymbol y}}
\def\bX{{\boldsymbol X}}
\def\bY{{\boldsymbol Y}}
\def\bZ{{\boldsymbol Z}}

\def\ds{\displaystyle}
\newtheorem{Theorem}{Theorem}[section]
\newtheorem{Definition}[Theorem]{Definition}
\newtheorem{Proposition}[Theorem]{Proposition}

\newtheorem{Corollary}[Theorem]{Corollary}

\begin{document}

%\pagestyle{plain}
%\title{MFGf\\ (Draft)}

\title{The convergence problem in mean field games with local coupling\\
%POUR AMO (F. Delarue)
}
\author{P. Cardaliaguet\thanks{Universit\'e Paris-Dauphine, PSL Research University, CNRS, Ceremade, 75016 Paris, France. cardaliaguet@ceremade.dauphine.fr}}

\maketitle

\begin{abstract} The paper studies the convergence,  as $N$ tends to infinity, of a system of $N$ coupled Hamilton-Jacobi equations, the Nash system, when the coupling between the players becomes increasingly singular. The limit equation turns out to be a  Mean Field Game system with a local coupling. 
\end{abstract}

\tableofcontents

\bigskip
\bigskip
\bigskip
\bigskip

\section*{Introduction}

In this paper we investigate the convergence of the Nash system associated with a differential game to the mean field game (MFG) system as the number of players tends to infinity. In differential game theory, the Nash system associated with a $N-$player differential game is a coupled system of $N$ Hamilton-Jacobi equations. In our previous work \cite{CDLL}, co-authored with F.~Delarue, J.-M.~Lasry and P.-L.~Lions, we explained that the solution of the Nash system converges, as $N$ tends to infinity, to the solution of the MFG system, which consists in a coupling between an Hamilton-Jacobi equation and a Fokker-Planck equation. We proved the result under the key assumption that the ``coupling" between the equations is nonlocal and regularizing. In the present setting, we consider the case where this coupling is singular: in the Nash system, the payoff of a player depends in an increasingly singular way on the players which are very close to her. We prove that, in this case, the solution of the Nash system converges to a solution of the Nash system with a local coupling. 

To better explain what we have in mind, let us consider the Nash system 
\be\label{NashIntro}
\left\{ \begin{array}{l}
\ds - \partial_t v^{N,i} (t,\bx)-  \sum_{j=1}^N \Delta_{x_j}v^{N,i}(t,\bx) + H(x_i,  D_{x_i}v^{N,i}(t,\bx)) \\
\ds \qquad \qquad + \sum_{j\neq i}  D_pH (x_j, D_{x_j}v^{N,j}(t,\bx))\cdot D_{x_j}v^{N,i}(t,\bx)=
F^{N,i}(\bx) \\
\qquad \qquad \qquad \qquad \qquad \qquad \qquad \qquad {\rm in}\; [0,T]\times (\R^{d})^N,\\
\ds v^{N,i}(T,\bx)= G(x_i)
\qquad {\rm in }\;  (\R^{d})^N.
\end{array}\right.
\ee
In the above system, the $N$ unknown maps $v^{N,i}$ depend on time and space in the form $(t,\bx)$ with $\bx= (x_1,\dots, x_N)\in (\R^d)^N$. 
The data are the horizon $T$, the  Hamiltonian $H:\R^d\times \R^d\to \R$, the terminal condition $G:\R^d\to \R$ and  the map $F^{N,i}:(\R^d)^N\to \R$. The maps $(F^{N,i})_{i=1,\dots,N}$ are called the coupling functions because they are responsible of all the interactions between the equations. 

We are also interested in the associated  system of $N$ coupled stochastic differential equations (SDE): 
\be\label{e.optitraj}
dY_{i,t}= -D_pH\bigl(Y_{i,t}, Dv^{N,i}(t, \bY_t) \bigr){\rm dt} +\sqrt{2} dB^i_t, \qquad t\in [0,T],\; i\in \{1, \dots, N\}, 
\ee
where $(v^{N,i})$ is the solution to \eqref{NashIntro} and the $((B^i_t)_{t\in [0,T]})_{i=1,\dots,N}$ are $d-$dimensional independent Brownian motions. In the language of differential games, the map $v^{N,i}$ is the value function associated with player $i$, $i\in \{1, \dots, N\}$ while $(Y_{i,t})$ is her optimal trajectory.  \\

In order to expect a limit system, we suppose that the coupling maps $F^{N,i}$ enjoy the following symmetry property:  
$$
F^{N,i}({\bx}) = F^N(x_i, m^{N,i}_{\bx}),
$$
where $F^N:\R^d\times {\mathcal P}(\R^d)\to \R$ is a given map ($ {\mathcal P}(\R^d)$ being the set of Borel probability measures on $\R^d$) and $m^{N,i}_{\bx}=\frac{1}{N-1}\sum_{j\neq i} \delta_{x_j}$ is the empirical measure of all players but $i$. Note that this assumption means that the players are indistinguishable: for a generic player $i$, players $k$ and $l$ (for $k,l\neq i$) play the same role. Moreover, all the players have a cost function with the  same structure. This key conditions ensures that the Nash system enjoys strong symmetry properties. 

In contrast with \cite{CDLL}, where $F^N$ does not depend on $N$ and is regularizing with respect to the measure, we assume here that the $(F^N)$  become increasingly singular as $N\to +\infty$. Namely we suppose that there exists a smooth (local) map $F:\R^d\times [0,+\infty)\to \R$ such that
\be\label{eq:FNtoFIntro}
\lim_{N\to+\infty} F^N(x, m {\rm dx}) = F(x, m(x)), 
\ee
for any sufficiently smooth density $m$ of a measure $m(x) {\rm dx}\in {\mathcal P}(\R^d)$. This assumption, which is the main difference with \cite{CDLL},  is very natural in the context of mean field games. One expects (and we will actually prove) that the limit system is a MFG system with local interactions: 
\be\label{MFGintro}
\left\{ \begin{array}{l}
\ds - \partial_t u - \Delta u +H(x,Du)=F(x,m(t,x))\qquad {\rm in }\; [t_{0},T]\times \R^d,
\\
\ds \partial_t m - \Delta m -{\rm div}( m D_pH(x, D u))=0\qquad {\rm in }\; [t_{0},T]\times \R^d, \\
\ds u(T,x)=G(x), \; m(t_0, \cdot)=m_0 
\qquad {\rm in }\;  \R^d.
\end{array}\right.
\ee
This system---which enjoys very nice properties---has been very much studied in the literature: see for instance \cite{gomes2016regularity, LcoursColl, porretta2015weak}. Note that in these papers the terminal condition $G$ may also depend  on the measure. For technical reasons we cannot allow this dependence in our analysis.\\

To explain in what extend the local framework differs from the nonlocal one, let us recall the ideas of proof in this later setting. 
The main ingredient in \cite{CDLL} for the proof of the convergence is the existence of a classical solution to the so-called {\it master equation}.  When $F^{N,i}(\bx)= F(x_i, m^{N,i}_{\bx})$, where $F:\R^d\times {\mathcal P}(\R^d)\to \R$ is sufficiently smooth (at least continuous), the master equation takes the form of a transport equation stated on the space of probability measures: 
\be\label{MasterEqIntro}
%\begin{array}{l}
\left\{\begin{array}{l} 
\ds - \partial_t U  - \Delta_x U +H(x,D_xU) \\
\ds    \qquad  -\int_{\R^d} \dive_y \left[D_m U\right]\ d m(y)+ \int_{\R^d} D_m U\cdot D_pH(y,D_xU) \ dm(y) =F(x,m)\\
 \ds \qquad \qquad  \qquad {\rm in }\; [0,T]\times \R^d\times {\mathcal P}(\R^d),\\
 U(T,x,m)= G(x) \qquad  {\rm in }\;  \R^d\times {\mathcal P}(\R^d).\\
\end{array}\right.
%\end{array}
\ee
In the above equation, $U$ is a scalar function depending on $(t,x,m)\in [0,T]\times \R^d\times {\mathcal P}(\R^d)$, $D_m U$ denotes the derivative with respect to the measure $m$ (see Section \ref{sec.not}). The interest of the map $U$ is that the $N-$tuple $(u^{N,i})$, where
\be\label{defuNiIntro}
u^{N,i}(t,\bx) = U(t,x_i, m^{N,i}_{\bx}),\qquad \bx=(x_1, \dots, x_N) \in (\R^{d})^N,
\ee
is an approximate solution to the Nash system \eqref{NashIntro} which enjoys very good regularity properties. In \cite{CDLL} we use these regularity properties in a crucial way to prove the convergence of the Nash system. \\

When $F$ is a local coupling, i.e., $F(x,m)=F(x,m(x))$ for any absolutely continuous measure $m$, the meaning of the master equation is not clear: obviously one cannot expect $U$ to be a smooth solution to \eqref{MasterEqIntro}, if only because the coupling function $F$ blows up at singular measures. So the master equation should be defined on a subset of sufficiently smooth density measures. But then the definition of the map $u^{N,i}$ through \eqref{defuNiIntro} is dubious and,  even if such a definition could make sense, there would be no hope that the $u^{N,i}$ satisfy the regularity properties required in the computation of \cite{CDLL}.\\

Our starting point is the following easy remark: if the coupling $F^N$ remains sufficiently smooth for $N$ large, then the solution $v^{N,i}$ should eventually become close to the solution of the master equation associated with $F^N$---and thus to the solution of the MFG system with nonlocal coupling $F^N$. On the other hand, if $F^N$ is close to $F$ (thus becoming singular), the solution $(u^N,m^N)$ is also close to the solution $(u,m)$ of the MFG system with local coupling. Thus if $F^N$ converges very slowing to $F$ while remaining ``sufficiently" smooth one can expect a convergence of $v^{N,i}$ to $u$. The whole point consists in quantifying in a careful way this convergence. \\

To give a flavor of our results, let us discuss a particular case. Let us assume for a while that 
$$
F^N(x,m)= F(\cdot, \xi^{\ep_N}\ast m(\cdot))\ast \xi^{\ep_N},
$$
where $\xi^\ep(x)= \ep^{-d}\xi(x/\ep)$, $\xi$ being a symmetric smooth nonnegative kernel with compact support and $(\ep_N)$ is a sequence which converges to $0$. Let $v^{N,i}$ be the solution of the Nash system \eqref{NashIntro}. Given 
an initial condition $(t_0,m_0)\in [0,T]\times {\mathcal P}(\R^d)$,  let $(u,m)$ be the solution to the MFG system 
$$
\left\{ \begin{array}{l}
\ds - \partial_t u - \Delta u +H(x,Du)=F(x,m(t,x))\qquad {\rm in }\; [t_{0},T]\times \R^d,
\\
\ds \partial_t m - \Delta m -{\rm div}( m D_pH(x, D u))=0\qquad {\rm in }\; [t_{0},T]\times \R^d, \\
\ds u(T,x)=G(x), \; m(t_0, \cdot)=m_0 
\qquad {\rm in }\;  \R^d.
\end{array}\right.
$$
 In order to describe the convergence of $v^{N,i}$ to $u$, we reduce the function $v^{N,i}$ to a function of a single variable by averaging it against the measure $m_0$: for $i\in \{1, \dots, N\}$, let 
$$
w^{N,i}(t_0,x_i,m_0) := \inte\dots \inte v^{N,i}(t_0, \bx) \prod_{j\neq i}m_0({\rm dx}_j) \qquad {\rm where }\; \bx=(x_1,\dots, x_N).
$$ 
Corollary \ref{prop.ex} states that, if $\ep_N = \ln(N)^{-\beta}$ for some $\beta \in (0, (6d(2d+15))^{-1})$, then 
$$
\left\| w^{N,i}(t_0,\cdot, m_0)-u(t_0,\cdot)\right\|_{\infty} \leq A(\ln(N))^{-1/B},
$$
for some constants $A,B>0$. Moreover, we show that the optimal trajectories $(Y_{i,t})$ converge to the optimal trajectory $(\tilde X_{i,t})$ associated with the limit MFG system and establish a propagation of chaos property (Proposition \ref{prop.choasex}). 

For a general sequence of couplings $(F^N)$, converging to a local coupling $F$ as in \eqref{eq:FNtoFIntro}, our main result (Theorem \ref{thm:main}) states that, if the  regularity of $F^N$ does not deteriorate too fast, then $w^{N,i}$ converges to $u$ and the optimal trajectories converge as well.

 Let us finally point out that, in order to avoid issue related to boundary conditions  or problems at infinity, we will assume that the data are periodic in space, thus working on the torus $\T^d=\R^d/\Z^d$.\\

Mean field game theory started with the pioneering works by Lasry and Lions \cite{LL06cr1, LL06cr2, LL07mf} and Caines, Huang and Malhamé \cite{HCMieeeAC06, HCMieeeAC07, HCMieeeDC07, HCMjssc07, HCMieee2010}. These authors introduced the mean field game system and discussed its properties: in particular, Lasry and Lions introduced the fundamental monotonicity condition on the coupling functions. They also discussed the various types of MFG systems (with soft or hard coupling, with or without diffusion). 

The link between the MFG system (which can be seen as a differential game with infinitely many players) and the differential games with finitely many players has been the object of several contributions. Caines, Huang and Malhamé \cite{HCMjssc07}, and  Delarue and Carmona \cite{CaDe2} explained how to use the solution of the MFG system to build $\ep-$Nash equilibria  (in open loop form) for $N-$person games. The convergence of the Nash system remained a puzzling issue for some time. The first results in that direction go back to \cite{LL06cr1, LL07mf} (see also \cite{feleqi2013derivation}), in the ``ergodic case", where the Nash system becomes a system of $N$ coupled equation in dimension $d$ (and not $Nd$ as in our setting): then one can obtain estimates which allow to pass to the limit. Another particular case is obtained when one is interested in Nash equilibria in open loop form: Fischer \cite{Fi14} and Lacker \cite{Lc14} explained in what extend one can expect to obtain the MFG system at the limit. For the genuine Nash system \eqref{NashIntro}, a first breakthrough was achieved by Lasry and Lions (see the presentation in \cite{LcoursColl}) who formally explained the mechanism towards convergence {\it assuming} suitable {\it a priori} estimates on the solution. For that purpose they also introduced the master equation (equation \eqref{MasterEqIntro}) and described (mostly formally) its main properties. 

The rigorous derivation of Lasry and Lions ideas took some time. The existence of a classical solution to the master equation has been obtained by several authors in different frameworks (Buckdahn, Li, Peng and Rainer
\cite{BuLiPeRa} for the linear master equation without coupling,  Gangbo and Swiech \cite{GS14-2} for the master equation without diffusion and in short time horizon,  Chassagneux, Crisan and Delarue \cite{CgCrDe} for the first order master equation, Lions \cite{LcoursColl} for an approach by monotone operators). The most general result so far is obtained in \cite{CDLL}, where the master equation is proved to be well-posed even for problems with {\it common noise}.  The main contribution of  \cite{CDLL} is, however, the convergence of the Nash system: it is obtained as a consequence of the well-posedness of the master equation. The present paper is the first attempt to show the convergence for a coupling which becomes singular. \\

The paper is organized in the following way: we first state our main notation (in particular for the derivatives with respect to a measure) and main assumptions. In section~2, we prove our key estimates on the solution of the master equation and on the MFG systems. The whole point is to display the dependence of the estimates with respect to the regularity of the coupling. The last part collects our convergence results. \\

{\bf Acknowledgement:} The author was partially supported by the ANR (Agence Nationale de la Recherche) project ANR-16-CE40-0015-01.
The author wishes to thank the anonymous referee for the very careful reading and for finding a serious gap in the previous version of the paper.

%%%%%%%%%%%%%%%%%%%%%%
\section{Notation and Assumptions}\label{sec.not}

For the sake of simplicity, the paper is written under the assumption that all maps are periodic in space. So the underlying state space is the torus $\T^d=\R^d/\Z^d$. This simplifying assumption allows to discard possible problems at infinity (or at the boundary of a domain).

\subsection{Notation}

We will need the following notations for the derivatives in space or in time of a map. 

For $u=u(x)$ and $l=(l_1,\dots,l_d)\in \N^d$, we denote by $D^lu(x)$ the derivative 
$\ds D^lu(x)= \frac{\partial^{|l|}u(x)}{\partial x_1^{l_1}\dots \partial x_d^{l_d}},$
where $|l|=l_1+\dots l_d$. If $k\in \N$, $D^ku(x)$ denotes the collection of derivatives $(D^lu(x))_{|l|= k}$. For $k=1$ and $k=2$, $Du(x)$ and $D^2u(x)$ denote the gradient and the Hessian of $u$ at $x$.

For $k\in \N$ and $\alpha\in (0,1)$, we denote by $C^{k+\alpha}$ the set of maps $u=u(x)$ which are of class $C^k$ and such that $D^lu$ is $\alpha-$H\"{o}der continuous for any $l\in \N^d$ with $|l|=k$. We set 
$$
\|u\|_{k+\alpha}:= \sum_{|l|\leq k} \|D^l u\|_\infty+ \sum_{|l|\leq k} \sup_{x\neq y} \frac{|D^lu(x)-D^lu(y)|}{|x-y|^\alpha}.
$$
 When a map $u$ depends on several space variables, say 2 for instance, we set in the same way 
 $$
\|u\|_{k+\alpha,k'+\alpha}:= \sum_{|l|\leq k, \ |l'|\leq k'} \|D^{l,l'}_{x,x'} u\|_\infty+ \sum_{|l|\leq k, |l'|\leq k'} \sup_{(x,x')\neq (y,y')} \frac{|D^{l,l'}_{x,x'}u(x,x')-D^{l,l'}_{x,x'}u(y,y')|}{|(x,x')-(y,y')|^\alpha}.
$$
For $p\geq 1$, the $L^p$ norm of $u$ is denoted by $\|u\|_{L^p}$. However, by abuse of notation, we denote by $\|u\|_\infty$ the $L^\infty$ norm of $u$ (instead of $\|u\|_{L^\infty}$). When $\phi$ is a distribution, we set 
$$
\|\phi\|_{-(k+\alpha)} := \sup_{\|u\|_{k+\alpha}\leq 1} |\phi(u)|. 
$$

Finally, when $u=u(t,x)$ is also time dependent, we denote by $\partial_t u$ the time derivative of $u$ and, as previously, by $D^lu$ its space derivative of order $l\in\N^d$. If $\alpha\in (0,1)$, we say that $u$ is in $C^{\alpha/2,\alpha}$ if
$$
\|u\|_{C^{\alpha/2,\alpha}}:=\|u\|_\infty+ \sup_{(t,x),(t',x')} \frac{|u(t,x)-u(t',x')|}{|x-x'|^\alpha+|t-t'|^{\alpha/2}} <+\infty.
$$
 We say that $u$ is  in $C^{1+\alpha/2,2+\alpha}$ if $\partial_t u$ and $D^2u$ belong to $C^{\alpha/2,\alpha}$. 

%%%%%%%%%%%%%%%%%%%%%%%%%%%%%
\subsection{Derivatives with respect to the measure}

We follow here  \cite{CDLL}. We denote by $\Pk$ the set of Borel probability measures on the torus $\T^d:=\R^d/\Z^d$. It is endowed with the Monge-Kantorovitch distance: 
$$
\dk(m,m') =\sup_\phi \inte \phi(y)\ d(m-m')(y),
$$
where the supremum is taken over all $1-$Lipschitz continuous maps $\phi:\T^d\to\R$.

\begin{Definition}\label{def:Diff} Let  $U:\Pw\to \R$  be a map. 
\begin{itemize} 
\item We say that $U$ is $C^{1}$ if there exists a continuous  map $\ds \frac{\delta U}{\delta m}:\Pw\times \T^d\to \R$ such that, for any $m,m'\in \Pk$,  
$$
U(m')-U(m) = \int_0^1\inte \frac{\delta U}{\delta m}((1-s)m+sm',y)\ d(m'-m)(y)ds.
$$
The map $\frac{\delta U}{\delta m}$ being defined up to an additive constant, we adopt the normalization convention
\be\label{ConvCondDeriv}
\inte \frac{\delta U}{\delta m}(m,y)dm(y)=0.
\ee
\item  If $\ds \frac{\delta U}{\delta m}$ is of class $C^1$ with respect to the second variable, 
the intrinsic derivative  $D_mU:\Pw\times \T^d\to \R^d$ is defined by 
$$
D_mU(m,y):= D_y \frac{\delta U}{\delta m}(m,y).
$$
\end{itemize}
\end{Definition}
It is know \cite{CDLL} that the map  $D_mU$ measures the Lipschitz regularity of $U$: 
$$
\ds \left|U(m')-U(m)\right|\; \leq \; \ds  \sup_{m''}\|D_mU(m'',\cdot)\|_\infty \dk(m,m')\qquad \forall m,m'\in \Pk.
$$

We will also need second order derivatives with respect to the measure. If $U$  and $\frac{\delta U}{\delta m}$ are of class $C^1$ with respect to the measure $m$, we denote by $\frac{\delta^2 U}{\delta m^2}:\Pk\times \T^d\times \T^d\to \R$ the derivative of $\frac{\delta U}{\delta m}$ with respect to $m$.  If $\frac{\delta^2 U}{\delta m^2}$ is sufficiently smooth, we also set 
$$
D^2_{mm} U(m,y,y')= D^2_{y,y'} \frac{\delta^2 U}{\delta m^2}(m,y,y'). 
$$

%%%%%%%%%%%%%%%%%%%%%%%
\subsection{Assumption}

Throughout the paper, we suppose that the following conditions are in force. 
\begin{itemize}
\item  The Hamiltonian $H:\T^d\times \R^d\to \R$ is smooth, globally Lipschitz continuous in both variables and locally uniformly convex with respect to the second variable: 
\be\label{Hunifell}
D^2_{pp}H(x,p)>0\qquad \forall (x,p)\in \T^d\times \R^d.
\ee
Moreover, $D_pH$ and its derivatives are globally Lipschitz continuous. 

\item $F:\T^d\times [0,+\infty)\to \R$ is smooth, with bounded derivatives in both variables. Moreover, $F$ is increasing with respect to the second variable, with $\ds \frac{\partial F}{\partial m}\geq \delta>0$ for some $\delta>0$ (note that  $\frac{\partial F}{\partial m}$ stands for the usual derivative of the map $F=F(x,m)$ with respect to the last variable). 

\item The terminal cost $G:\T^d\to \R$ is a smooth map. 

\item For any $N\in \N$, $F^N:\T^d\times \Pk\to \R$ is monotone: 
$$
\int_{\T^d} (F^N(x,m)-F^N(x,m'))d(m-m')(x)\geq 0\qquad \forall m,m'\in \Pk. 
$$

\item (difference between $F^N$ and $F$) For any $R>0$ and $\alpha\in (0,1)$, there exists $k^{R,\alpha}_N$ such that $k^{R,\alpha}_N\to 0$ as $N\to+\infty$ and
\be\label{ecartFFep}
\| F^N(\cdot, m{\rm dx})-F(\cdot, m(\cdot))\|_\infty \leq  k^{R,\alpha}_N,
\ee
for any density $m$ such that $\|m\|_{C^\alpha}\leq R$. 

\item (uniform regularity of $F^N$) For any $R>0$ and $\alpha\in (0,1)$, there exists $\kappa_{R,\alpha}>0$ such that, for any $N\in \N$, 
\be\label{H.unifregu}
\ds \left| F^N(x,m{\rm dx})-F^N(y, m'{\rm dx})\right| \leq \kappa_{R,\alpha}\left( |x-y|^\alpha+ \|m-m'\|_{\infty}\right)
\ee
for any density $m,m'$ with $\|m\|_{C^\alpha}, \|m'\|_{C^\alpha}\leq R$. 

\item (regularity assumptions on $F^N$) For any $N\in \N$, $F^N$ is of class $C^2$ with respect to the $m$ variable and, for any $\alpha\in (0,1)$, there exists a constant $K_{N,\alpha}$ such that
\be\label{defKep1}
\left\| F^N(\cdot,m)\right\|_{4+\alpha}+\left\| \frac{\delta F^N}{\delta m}\bigl(\cdot, m, \cdot \bigr)\right\|_{(4+\alpha,4+\alpha)}+
 \left\| \frac{\delta^2 F^N}{\delta m^2}(\cdot, m,\cdot,\cdot)\right\|_{(4+\alpha,4+\alpha,4+\alpha)}
 \leq K_{N,\alpha},
\ee
for any $m\in \Pk$. 
%\be\label{defKep2}
% \left\| \frac{\delta^2 F^N}{\delta m^2}(\cdot, m,\cdot,\cdot)\right\|_{(4+\alpha,4+\alpha,4+\alpha)}
%\leq K_{N,\alpha} 
%\ee
\end{itemize}
%COMMENTAIRE : on a besoin (au moins) de $HF1(n+1)$ avec $n\geq 2$, soit $HF1(3)$. 

Some comments are in order. The Lipschitz regularity condition of the Hamiltonian $H$ is not very natural in the context of MFG, but we do not know how to avoid it: if it is probably not necessary in the estimates of Section \ref{sec:reguestim}, it is required for the convergence of the Nash system. Let us just note  that it simplifies at lot the existence of solutions for the Nash system (see, for instance, \cite{LSU}) as well as for the limit MFG system (see \cite{LL07mf}): indeed, without the assumption that $D_pH$ is bounded, existence of classical solution to \eqref{MFGintro} is related on a subtle interplay between the growth of $H$ and of $F$. Note however that \cite{carmona2016probabilistic} overcomes this issue for nonlocal couplings functions. 

The monotonicity of $F^N$ and $F$ are natural to ensure the uniqueness of the solution to the respective MF systems. However, the strict monotonicity of $F$ is unusual: it is used here to give a good control between the MFG limit system \eqref{MFGintro} and the (nonlocal) MFG system with coupling term $F^N$ (given in \eqref{e.MFG.t0m0} below): see Proposition \ref{prop.estiuep-u}. 

In MFG problems, one often assumes that the terminal  cost also depends on the measure. It seems difficult to allow this dependence in our context, since in this case the uniform regularity of the solution of the MFG system (and hence of the master equation)  near time $T$ could be lost (see Proposition \ref{prop.estiuep-u}).

Note that $K_{N,\alpha}\to+\infty$ as $N\to+\infty$ because $F^N(x,m)$ blows up if $m$ is a singular measure. So $F^N$ becomes closer and closer to $F$ for smooth densities while its regularity at general probability measures deteriorates. However, assumption \eqref{H.unifregu} states that the $F^N$ are uniformly H\"{o}der continuous when evaluated at probability densities which are H\"{o}der continuous. Assumption \eqref{defKep1} explains how fast the regularity of $F^N$ degrades as $N\to+\infty$. One could have had different constants  for $F^N$, $\frac{\delta F^N}{\delta m}$ and $\frac{\delta^2 F^N}{\delta m^2}$: the choice to have a unique constant is only made for simplicity of presentation.

%%%%%%%%%%%%%%%%%%%%%%%%
\subsection{Main example}

Here is a typical example for $F^N$ when $F$ satisfies our standing conditions. We assume that $F^N$ is of the form $F^N=F^{\ep_N}$ where $(\ep_N)$ is a positive sequence which tends to $0$ and 
\be\label{Fepex0}
F^\ep(x,m)= \left( \xi^\ep \ast F(\cdot, \xi^\ep\ast m(\cdot))\right)(x),
\ee
with $\xi^\ep(x)= \ep^{-d}\xi(x/\ep)$, $\xi$ being a symmetric smooth nonnegative kernel with compact support. This example was introduced in \cite{LcoursColl}. 

\begin{Proposition}\label{prop.ex0} If $F^N$ is defined by \eqref{Fepex0}, then $F^N$ is monotone and satisfies \eqref{H.unifregu}.
Moreover the constants $k^{R,\alpha}_N$ and $K_{N,\alpha}$ associated with $F^N$ as in \eqref{ecartFFep} and in \eqref{defKep1} can be estimated by
\be\label{estikKalpha0}
k^{R,\alpha}_N\leq C(1+R) \ep_N^\alpha, \qquad K_{N,\alpha}\leq C\ep_N^{-2d-12-3\alpha},
\ee
where $C$ depends on  $F$ and $\xi$. 
\end{Proposition}

\begin{proof} Under the monotonicity assumption on $F$, it is known that  the $F^\ep$ are monotone (see \cite{LcoursColl}).  Next we prove that the $F^\ep$ satisfy \eqref{H.unifregu}. Let $m {\rm dx},m'{\rm dx}\in \Pk$ with $\|m\|_{C^\alpha},\|m'\|_{C^\alpha}\leq R$. Then 
$$
\begin{array}{l}
\ds \left| F^\ep(x,m{\rm dx})-F^\ep(x', m'{\rm dx})\right|\\
\qquad  \ds \leq \sup_{y} \left| F(x-y, m\ast \xi^\ep(x-y))-F(x'-y, m'\ast \xi^\ep(x'-y))\right|\\
\qquad \ds \leq C\sup_y \left[|x-x'| + |m\ast \xi^\ep(x-y)-m'\ast \xi^\ep(x'-y)|\right] \\
\qquad \ds \leq C\sup_y \left[|x-x'| + R |x-x'|^\alpha+ |m\ast \xi^\ep(x'-y)-m'\ast \xi^\ep(x'-y)|\right]  \\
\qquad \ds \leq C \left[R |x-x'|^\alpha+ \|m-m'\|_\infty\right].
\end{array}
$$
We now estimate the constants $k^{R,\alpha}_N$ and $K_{N,\alpha}$.
It is  enough to estimate $k^{R,\alpha}_\ep$ and $K_{\ep,\alpha}$, where 
$$
k^{R,\alpha}_\ep:=\sup_m \| F^\ep(\cdot, m{\rm {\rm dx}})-F(\cdot, m(\cdot))\|_\infty,
$$ 
the supremum being taken over the densities $m$ such that $\|m\|_{C^\alpha}\leq R$, and 
$$
K_{\ep,\alpha}=\max\{K^{(0)}_{\ep,\alpha}, K^{(1)}_{\ep,\alpha}, K^{(2)}_{\ep,\alpha}\},
$$ with
$$
K^{(0)}_{\ep,\alpha}:= \sup_{m\in \Pk} \left\| F^\ep\bigl(\cdot, m \bigr)\right\|_{4+\alpha} , \qquad 
%$$
%$$
K^{(1)}_{\ep,\alpha}:= \sup_{m\in \Pk} \left\| \frac{\delta F^\ep}{\delta m}\bigl(\cdot, m, \cdot \bigr)\right\|_{(4+\alpha,4+\alpha)} 
$$
and 
$$
K^{(2)}_{\ep,\alpha}:= \sup_{m\in \Pk}  \left\| \frac{\delta^2 F^\ep}{\delta m^2}(\cdot, m,\cdot,\cdot)\right\|_{(4+\alpha,4+\alpha,4+\alpha)}.
$$
%Recall that, if $\phi$ is H\"{o}der continuous with $\|\phi\|_{C^\alpha}\leq R$, then 
%$$
%\left|\xi^\ep\ast \phi(x)-\phi(x)\right|\leq R\ep^\alpha \int_{\R^d} \xi(y)|y|^\alpha. 
%$$
Let $m\in \Pk$ be such that $\|m\|_\alpha\leq R$ and $x\in \T^d$. As $\|m\|_\alpha\leq R$, we have: 
$$
\|\xi^\ep\ast m-m\|_\infty\leq R\ep^\alpha\int_{\R^d} \xi(y)|y|^\alpha\leq CR\ep^\alpha,
$$
so that 
$$
\|F(\cdot,\xi^\ep\ast m(\cdot))-F(\cdot,m(\cdot))\|\leq CR\ep^\alpha,
$$
because $F$ is Lipschitz continuous. Thus  
$$
\begin{array}{l}
\ds | F^\ep(x, m{\rm {\rm dx}})-F(x, m(x))|\\
\qquad \qquad  \leq   \ds | \xi^\ep\ast F(\cdot,\xi^\ep\ast m(\cdot))(x) - F(x,\xi^\ep\ast m(x))|+ 
 |F(x,\xi^\ep\ast m(x))- F(x,m(x))| 
  \leq  \ds CR\ep^\alpha.
\end{array}
$$
Therefore 
$$
k^{R,\alpha}_N\leq C(1+R) \ep^\alpha, 
$$
where $C$ depends on the Lipschitz constant of $F$ and on $\xi$. 

For any $l\in \N^d$, we have 
$$
D^l_x F^\ep(x,m) = \int_{\R^d} F(y,\xi^\ep\ast m(y)) D^l\xi^\ep(x-y)dy . 
$$
Hence 
$$
\left| D^l_x F^\ep(x,m) - D^l_x F^\ep(x',m) \right| \leq 
\int_{\R^d} \left| F(y,\xi^\ep\ast m(y))\right| \ \left| D^l\xi^\ep(x-y)- D^l\xi^\ep(x'-y)\right| \ dy.
$$
As $F$ is Lipschitz continuous, it has a linear growth: 
$$
|F(y,\xi^\ep\ast m(y))|\leq C(1+ \xi^\ep\ast m(y)).
$$
Thus, as $\xi^\ep$ has a support with a uniformly bounded diameter, 
$$
\|D^l_x F^\ep(\cdot,m) \|_\alpha \leq C  \|D^l\xi^\ep(\cdot)\|_\alpha (1+ \int_{\R^d} |\xi^\ep\ast m(y)|  dy)\leq 
C \ep^{-(d+|l|+\alpha)}.
$$
So
$$
K^{(0)}_{\ep,\alpha}\leq C \ep^{-(d+4+\alpha)}.
$$

On the other hand, 
$$
\frac{\delta F^\ep}{\delta m}(x,m,z)= \int_{\R^d} \sum_{k\in \Z^d} \xi^\ep(y-z-k)\frac{\partial F}{\partial m}(y, \xi^\ep\ast m(y))\xi^\ep(x-y)dy
$$
and 
$$
\frac{\delta^2 F^\ep}{\delta m^2}(x,m,z,z')= \int_{\R^d} \sum_{k,k'\in \Z^d} \xi^\ep(y-z-k)\xi^\ep(y-z'-k')\partial^2_{mm}F(y, \xi^\ep\ast m(y))\xi^\ep(x-y)dy.
$$
Hence, for any $l,l'\in \N^d$,  
$$
\left|D^{l,l'}_{x,z}\frac{\delta F^\ep}{\delta m}(x,m,z)\right|\leq C
 \int_{\R^d} \sum_{k\in \Z^d} |D^{l'}\xi^\ep(y-z-k)|\  |D^l\xi^\ep(x-y)|dy,
$$
where $C$ is the Lipschitz constant of $F$. Thus, if $|l|,|l'|\leq 4$ and if the support of $\xi$ is contained in the ball $B_R$, we have 
$$
\left|D^{l,l'}_{x,z}\frac{\delta F^\ep}{\delta m}(x,m,z)\right|\leq C
\sum_{k\in \Z^d} \left\|D^{l'}\xi^\ep(\cdot-z-k)\right\|_{L^\infty(B_{R\ep}(x))}  \left\|D^l\xi^\ep(x-\cdot)\right\|_{L^1} \leq C\ep^{-d-8},
$$
where $B_{R\ep}(x)$ is the ball centered at $x$ and of radius $R\ep$. 
In the same way,
$$
\left\|\frac{\delta^2 F^\ep}{\delta m^2}(\cdot,m,\cdot,\cdot)\right\|_{4,4,4}\leq  C\ep^{-2d-12}.
$$
We can estimate in the same way the H\"{o}der norms of $\delta F^\ep/\delta m$ and  $\delta^2 F^\ep/\delta m^2$. 
\end{proof}

%%%%%%%%%%%%%%%%%%%%%%%%%%%%%%%%%%%%
\section{Regularity estimates}\label{sec:reguestim}

In this section, we prove estimates on the solutions of the MFG systems and on the solutions of the master equation with the smoothen coupling $F^N$: the whole point is to keep track of the dependence with respect to $N$ in these estimates. 

Let $U^N$ be the solution to the master equation 
\be\label{MasterEq}
%\begin{array}{l}
\left\{\begin{array}{l} 
\ds - \partial_t U^N(t,x,m)  - \Delta_x U^N +H(x,D_xU^N(t,x,m)) \\
\ds \qquad  -\int_{\R^d} \dive_y \left[D_m U^N(t,x,m,y)\right]\ d m(y) \\
\ds    \qquad + \int_{\R^d} D_m U^N(t,x,m,y)\cdot D_pH(y,D_xU^N) \ dm(y) =F^N(x,m)\\
 \ds \qquad \qquad  \qquad {\rm in }\; [0,T]\times \R^d\times {\mathcal P}(\T^d),\\
 U^N(T,x,m)= G(x) \qquad  {\rm in }\;  \R^d\times {\mathcal P}(\T^d).
\end{array}\right.
%\end{array}
\ee
The existence and the uniqueness of the classical solution $U^N$ to \eqref{MasterEq} are established in \cite{CDLL}. Namely:

\begin{Theorem}[\cite{CDLL}]\label{thm:master}  Assume that $F^N$, $G$ and $H$ satisfy  our standing assumptions. 
Then the first order master equation \eqref{MasterEq} has a unique classical solution which is of class $C^2$ with respect to the $m$ variable. 
\end{Theorem}

As the coupling $F^N$ becomes increasingly singular  as $N\to+\infty$, so does $U^N$. The next result collects the upper bounds on the derivatives of $U^N$.

\begin{Theorem}\label{them:KeyEsti} Under our standing assumptions, we have, for any $(t_0,m_0)\in [0,T]\times \Pk$ and for any $\alpha\in (0,1)$:
$$
\left\|U(t_0,\cdot,m_0)\right\|_{4+\alpha} \leq CK_{N,\alpha}^3,
$$
$$
 \left\| \frac{\delta U^N}{\delta m}(t_0,\cdot,m_0,\cdot)\right\|_{(k+\alpha, k+\alpha)} +
 \left\| D_m U^N(t_0,\cdot,m_0,\cdot)\right\|_{(k+\alpha, k-1+\alpha)}  \leq  CK_{N,\alpha}^{3k-2},
 $$
if $k\in \{1,\dots, 4\}$, and,  if $k=2,3$, 
$$
\left\| \frac{\delta^2 U^N}{\delta m^2}(t_0,\cdot,m_0,\cdot,\cdot)\right\|_{k+\alpha,k-1+\alpha,k-1+\alpha}
+\left\| D^2_{mm} U^N(t_0,\cdot,m_0,\cdot,\cdot)\right\|_{k+\alpha,k-2+\alpha,k-2+\alpha}
\leq C K_{N,\alpha}^{12k},
$$
where $C$  depends on $\alpha$ and on the data but not on $N$, $t_0$ or $m_0$. 
\end{Theorem}

As the bound on the derivative $D_mU^N$  (respectively $D^2_{mm}U^N$) provides a bound on the Lipschitz continuity of $U^N$ (respectively $D_mU^N$), we have:
$$
\left\|U^N(t, \cdot,m_1)-U^N(t, \cdot,m_2)\right\|_{k+\alpha}\leq CK_{N,\alpha}^{3k-2} \dk(m_1,m_2)
$$
and
$$
\left\|D_mU^N(t, \cdot,m_1,\cdot)-D_mU^N(t, \cdot,m_2,\cdot)\right\|_{k+\alpha,k-2+\alpha}\leq CK_{N,\alpha}^{12k} \dk(m_1,m_2),
$$
for any $m_1,m_2\in \Pk$ and $k=2,3$. 

The proof of Theorem \ref{thm:master} consists in estimating carefully the various steps in the construction of $U^N$ in \cite{CDLL}. It is given through a series of statements: Proposition \ref{prop:reguum} for the space regularity of $U^N$, Corollary \ref{cor:estiDmU}  for the bound on $\frac{\delta U^N}{\delta m}$ and Corollary \ref{cor:delta2U} for the  estimate on $\frac{\delta^2 U^N}{\delta m^2}$.

Let us recall \cite{CDLL} that the map $U^N$ is given by the representation formula: 
$$
U^N (t_0,x,m_0) = u^N(t_0,x)
$$
for any $(t_0,x,m_0)\in [0,T]\times \T^d\times {\mathcal P}(\T^d)$, where $(u^N,m^N)$ is the unique solution to the MFG system
\be\label{e.MFG.t0m0}
\left\{ \begin{array}{l}
\ds - \partial_t u^N - \Delta u^N +H(x,Du^N)=F^N(x,m^N(t))\qquad {\rm in }\; [t_{0},T]\times \T^d,
\\
\ds \partial_t m^N - \Delta m^N -{\rm div}( m^N D_pH(x, D u^N))=0\qquad {\rm in }\; [t_{0},T]\times \T^d, \\
\ds u^N(T,x)=G(x), \; m^N(t_0, \cdot)=m_0 
\qquad {\rm in }\;  \T^d.
\end{array}\right.
\ee

%%%%%%%%%%%%%%%%%%%%%%%%%%%
%%%%%%%%%%%%%%%%%%%%%%%%%%%%%

%%%%%%%%%%%%%%%%%%%%%%%%%
\subsection{Estimates on the MFG systems for smooth initial conditions}

Fix an initial condition $(t_0,m_0)\in [0,T]\times \Pk$. We consider the MFG system: 
\be\label{e.MFGsystLoc}
\left\{ \begin{array}{l}
\ds - \partial_t u - \Delta u +H(x,Du)=F(x,m(t,x))\qquad {\rm in }\; [t_0,T]\times \T^d,
\\
\ds \partial_t m - \Delta m -{\rm div}( m D_pH(x, D u))=0\qquad {\rm in }\; [t_0,T]\times \T^d, \\
\ds u(T,x)=G(x), \; m(t_0, \cdot)=m_0 
\qquad {\rm in }\;  \T^d,
\end{array}\right.
\ee
and compare its solution with the solution to the MFG system \eqref{e.MFG.t0m0}. 

\begin{Proposition}\label{prop.estiuep-u} Under our standing assumptions, let $t_0\in [0,T]$, $m_0\in\Pk$ be a positive density of class $C^{\alpha}$ (where $\alpha\in (0,1)$) and 
$(u^N,m^N)$ and $(u,m)$ be the solution to the MFG systems \eqref{e.MFG.t0m0} and \eqref{e.MFGsystLoc} respectively. Then there exists $\beta\in (0,\alpha]$ such that the $(u^N,m^N)$ are bounded in $C^{1+\beta/2,2+\beta}\times C^{\beta/2,\beta}$  independently of $N$. Moreover, 
\be\label{esti.dm}
\forall t,t'\in [t_0,T], \qquad \ds \dk(m^N(t),m^N(t'))+ \dk(m(t),m(t'))\leq C |t-t'|^{1/2}
\ee
and 
\be\label{esti-ecartuuep}
\sup_{t\in [0,T]}\|u^N(t,\cdot)-u(t,\cdot)\|_{H^1(\T^d)}+ \|m^N-m\|_{L^2}\leq Ck^{R,\alpha}_N,
\ee
where the constants $C$ and $R$ depend on the data and on $m_0$, but not on $N$. In particular, 
$$
\sup_{t\in [0,T]}\|u^N(t,\cdot)-u(t,\cdot)\|_{W^{1,\infty}} \leq  C\left(k^{R,\alpha}_N\right)^{\tfrac{2}{(d+2)}}.
$$
\end{Proposition}

\begin{proof} Existence of a solution to \eqref{e.MFG.t0m0} and to \eqref{e.MFGsystLoc} is well-known: see \cite{LL07mf}. Estimates \eqref{esti.dm} is a known consequence of the $L^\infty$ bound on $D_pH$. 

We now check the regularity of $m^N$. As $D_pH$ is bounded and $m_0$ is in $C^\alpha$, standard estimates for parabolic equations in divergence form (Theorem III.10.1 of \cite{LSU}) state that the $m^N$ are bounded in $C^{\beta/2, \beta}$ for some $\beta\in (0,\alpha]$. Note that the bound and $\beta$ depend on $\alpha$, $\|D_pH\|_\infty$ and $\|m_0\|_{C^\alpha}$ only. 

We now plug this estimate into the parabolic equation for $u^N$. For this we note that the map $(t,x)\to F^N(x, m^N(t))$ is uniformly H\"{o}der continuous. Indeed,  in view of assumption \eqref{H.unifregu} and the uniform regularity of $m^N$, 
$$
\begin{array}{rl}
\ds \left|F^N(x, m^N(t))-F^N(x', m^N(t'))\right|\;   \leq  & \ds \kappa_{R,\beta} \left( |x-x'|^\beta+ \|m^N(t,\cdot)-m^N(t',\cdot)\|_\infty\right) \\
\leq & \ds \kappa_{R,\beta} \left( |x-x'|^\beta+|t-t'|^{\beta/2}\right)
\end{array}
$$
where $R:=\sup_N \|m^N\|_{C^{\beta/2,\beta}}+\|m\|_{C^{\beta/2,\beta}}$. 
Since the terminal condition $G$ is $C^{2+\beta}$ and is independent of $m^N$ and since $H$ is Lipschitz continuous, standard estimates on Hamilton-Jacobi equations imply that the $u^N$ are bounded in $C^{1+\beta/2,2+\beta}$. \\

We now establish  \eqref{esti-ecartuuep}: following \cite{LL06cr2, LL07mf}, we have 
$$
\begin{array}{l}
\ds  \left[\int_{\T^d} (u^N-u)(m^N-m)\right]_0^T \\
\qquad \ds = -\int_0^T\int_{\T^d} m \left(H(x,Du^N)-H(x,Du)- D_pH(x,Du)\cdot D(u^N-u) \right) \\
\qquad \qquad \ds -\int_0^T\int_{\T^d} m^N \left(H(x,Du)-H(x,Du^N)-  D_pH(x,Du^N)\cdot D(u-u^N) \right) \\
\qquad \qquad \ds - \int_0^T\int_{\T^d} (F^N(x,m^N(t))-F(x,m(t,x))(m^N(t,x)-m(t,x)).
\end{array}$$
Note that, on the one hand,  $m^N(0)=m(0)=m_0$ and $u^N(T)=u(T)=G$. So the left-hand side vanishes. On the other hand, by strong maximum principle, $m$ is bounded below by a positive constant since $m_0$ is positive. As the $u^N$ and $u$ are uniformly Lipschitz continuous and assumption \eqref{Hunifell} holds,  we obtain:
$$
\begin{array}{rl}
\ds C^{-1}\int_0^T\int_{\T^d} |Du^N-Du|^2 \;  \leq & \ds   \ds - \int_0^T\int_{\T^d} (F^N(x,m^N(t))-F(x,m(t,x))(m^N(t,x)-m(t,x)).
\end{array}
$$
As $F=F(x,m)$ is increasing in the second variable with $\frac{\partial F}{\partial m}\geq \delta$ and as assumption \eqref{ecartFFep} holds, we have:
$$
\begin{array}{l}
 \ds  \int_0^T\int_{\T^d} (F^N(x,m^N(t))-F(x,m(t,x)))(m^N(t,x)-m(t,x)) \\ 
 \qquad =  \ds \int_0^T\int_{\T^d} (F^N(x,m^N(t))-F(x,m^N(t,x)))(m^N(t,x)-m(t,x))\\
 \qquad \qquad \ds + \int_0^T\int_{\T^d} (F(x,m^N(t,x))-F(x,m(t,x)))(m^N(t,x)-m(t,x))\\
\qquad \geq  \ds - Ck^{R,\beta}_N\|m^N-m\|_{L^1}  + \delta  \int_0^T\int_{\T^d} (m^N(t,x)-m(t,x))^2.
\end{array}
$$
We obtain therefore 
$$
\ds C^{-1}\int_0^T\int_{\T^d} |Du^N-Du|^2+ \delta  \int_0^T\int_{\T^d} (m^N(t,x)-m(t,x))^2 \leq Ck^{R,\beta}_N \|m^N-m\|_{L^1}\leq Ck^{R,\beta}_N \|m^N-m\|_{L^2}.
$$
Hence 
$$
\|Du^N-Du\|_{L^2}+ \|m^N-m\|_{L^2}\leq Ck^{R,\beta}_N.
$$
In particular 
$$
\begin{array}{l}
\ds \|F^N(\cdot, m^N(t))-F(\cdot, m(t,\cdot))\|_{L^2}\\
\qquad \ds  \leq \|F^N(\cdot, m^N(t))-F(\cdot, m^N(t,\cdot))\|_\infty+ \|F(\cdot, m^N(t,\cdot))-F(\cdot, m(t,\cdot))\|_{L^2} \\
\qquad \ds \leq C k^{R,\beta}_N + C \|m^N- m\|_{L^2} \;\leq \; Ck^{R,\beta}_N.
\end{array}
$$
Therefore the difference $w:= u^N-u$ satisfies 
$$
-\partial_t w-\Delta w = h(t,x)
$$
with $h(t,x)= F^N(x, m^N(t))-F(x, m(t,x))-H(x,Du^N(t,x))+H(x,Du(t,x))$. By our previous bounds, we have 
$\|h\|_{L^2}\leq C k^{R,\beta}_N$, so that standard estimates on the heat equation imply that 
$$
\sup_{t\in [t_0,T]} \|u^N(t,\cdot)-u(t,\cdot)\|_{H^1(\T^d)}\leq Ck^{R,\beta}_N.
$$
As, for any smooth map $\phi:\T^d\to \R$,  one has:
$\ds
\|\phi\|_\infty\leq C \|\phi\|_{L^2}^{\frac{2}{d+2}}\|D\phi\|_\infty^{\frac{d}{d+2}},$
and since $u^N$ and $u$ are bounded in $C^{1+\beta/2,2+\beta}$, we get
$$
\|u^N-u\|_\infty+ \|Du^N-Du\|_\infty \leq C \left(k^{R,\beta}_N\right)^{\frac{2}{d+2}}.
$$
We conclude by recalling that $k^{R,\beta}_N\leq k^{R,\alpha}_N$. 
\end{proof}

%NOTE : $D^2_{pp}H\geq CI_d$ INUTILE A CE STADE. Mais utile partout ailleurs\\

A straightforward consequence of Proposition \ref{prop.estiuep-u} is the following estimate on optimal trajectories associated with the MFG systems \eqref{e.MFG.t0m0} and \eqref{e.MFGsystLoc}. 

\begin{Corollary}\label{cor:solEDS} Let $m_0\in \Pk$,  $(u^N,m^N)$ and $(u,m)$ be the solution to the MFG system \eqref{e.MFG.t0m0} and \eqref{e.MFGsystLoc} respectively.  Let $t_0\in [0,T)$ and $Z$ be a random variable independent of a Brownian motion $(B_t)$. If $(X_t)$ and $(X^N_t)$ are the solution to 
$$
\left\{\begin{array}{l}
dX_t= -D_pH(X_t,Du(t,X_t)){\rm dt} +\sqrt{2}dB_t \qquad {\rm in}\; [t_0,T],\\
X_{t_0}=Z,
\end{array}\right.
$$
and 
$$
\left\{\begin{array}{l}
dX^N_t= -D_pH(X_t,Du^N(t,X^N_t)){\rm dt} +\sqrt{2}dB_t \qquad {\rm in}\; [t_0,T],\\
X^N_{t_0}=Z, 
\end{array}\right.
$$
then 
$$
\E\left[\sup_{t\in  [t_0,T]} \left|X_t-X^N_t\right|\right]\leq C\left(k^{R,\alpha}_N\right)^{\frac{2}{d+2}},
$$
where $C$ and $R$ are as in Proposition \ref{prop.estiuep-u}.  
\end{Corollary}

\begin{proof} By Proposition \ref{prop.estiuep-u}, we have 
$$
\|Du^N-Du\|_\infty \leq C \left(k^{R,\alpha}_N\right)^{\frac{2}{d+2}}.
$$
The conclusion  follows easily since  $D_pH$ is Lipschitz continuous. 
\end{proof}

%%%%%%%%%%%%%%%%%%%%%%%%%
\subsection{Estimates on the MFG systems for general initial conditions}

We now establish regularity estimates for the MFG system which are valid for any initial conditions. 

\begin{Proposition}\label{prop:reguum} Let $U^N$ be the solution to the master equation \eqref{MasterEq}, $(u^N,m^N)$ be the solution to \eqref{e.MFG.t0m0} for an arbitrary initial condition $(t_0,m_0)\in [0,T]\times \Pk$. Then, for any $\alpha\in (0,1)$, we have  
$$
\sup_{t\in [t_0,T]} \|u^N(t)\|_{4+\alpha}\leq CK_{N,\alpha}^3,
$$
where $C$ depends on $\alpha$. 
In particular, 
$$
\sup_{t\in [0,T], m\in \Pk} \left\|U^N(t,\cdot,m)\right\|_{4+\alpha} \leq CK_{N,\alpha}^3.
$$
\end{Proposition}

\begin{proof} Assumption \eqref{defKep1} implies:  
$$
\sup_{t\in [0,T]}\left\|D^l F^N(\cdot,m^N(t))\right\|_\infty\leq  K_{N,\alpha}\qquad \mbox{\rm for any $l\in \N^d$, with $|l|\leq 4$.}
$$
By maximum principle we have
$$
\|u^N\|_\infty\leq C\left(\|H(\cdot,0)\|_\infty+\|G\|_\infty+\|F^N\|_\infty\right)\leq CK_{N,\alpha}.
$$
Standard Lipschitz estimates for Hamilton-Jacobi equations (with a globally Lipschitz continuous Hamiltonian $H$)  lead to 
$$
\|Du^N\|_\infty\leq C\left(1+\|DG\|_\infty+ \|D_xF^N\|_\infty\right) \leq C K_{N,\alpha}.
$$
For any  $l\in \N^d$ with $|l|=1$, the map $w_l:=D^l u^N$ solves the linear equation with bounded coefficient:
$$
\left\{\begin{array}{l}
\ds -\partial_t w_l-\Delta w_l+D_pH(x,Du^N)\cdot Dw_l= D^l_xF^N(x,m^N)-D^l_xH(x,Du^N)\qquad {\rm in}\; (t_0,T)\times \T^d,\\
\ds w_l(T,x)=D^lG(x)\qquad {\rm in}\; \T^d.
\end{array}\right.
$$
So, for any $\alpha\in (0,1)$, we have (Proposition \ref{prop.reguregu} in appendix)
$$
\sup_{t\in [t_0,T]} \|w_l(t)\|_{1+\alpha} \leq C \left[\left\|D^l_xF^N(\cdot,m^N)\right\|_\infty+ \left\|D^lG\right\|_{1+\alpha}+\left\|D^l_xH(\cdot,Du^N)\right\|_\infty\right]\leq CK_{N,\alpha}.
$$
This implies that 
$$
\sup_{t\in [t_0,T]} \|u^N(t)\|_{2+\alpha} \leq CK_{N,\alpha}.
$$
We now estimate the second order derivative. Let $w_l:= D^lu^N$ with  $l=l_1+l_2$, $l_1,l_2\in \N^d$ and $|l_1|=|l_2|=1$. Then $w_l$ solves  the linear equation with bounded coefficients:
$$
\left\{\begin{array}{l}
\ds -\partial_t w_l-\Delta w_l+D_pH\cdot Dw_l = D^l_{xx}F^N(x,m^N)-D^2_{xp}He_{l_1}\cdot Dw_{l_2}\\ 
\qquad -D^2_{xp}He_{l_2}\cdot Dw_{l_1}-D^2_{pp}HDw_{l_1}\cdot Dw_{l_2} -D^2_{xx}He_{l_1}\cdot e_{l_2}\; {\rm in}\; (t_0,T)\times \T^d,\\
\ds w_l(T,x)=D^lG(x)\qquad {\rm in}\; \T^d,
\end{array}\right.
$$
(where $H=H(x,Du^N)$) so that 
$$
\begin{array}{rl}
\ds \sup_{t\in [t_0,T]}\|w_l(t)\|_{1+\alpha}\; \leq & \ds C \left[ \|D^2_{xx}F^N\|_\infty+\|D^2_{xp}H\|_\infty \|D^2u^N\|_\infty+
\|D^2_{pp}H\|_\infty\|D^2u^N\|^2_\infty\right.\\
& \qquad \ds \left. +
\|D^2_{xx}H(\cdot,Du^N)\|_\infty+ \|D^lG(\cdot)\|_{1+\alpha}\right] \leq \ds CK_{N,\alpha}^2,
\end{array}
$$
where we used  our previous estimate on $\sup_t \|u^N(t)\|_{2+\alpha}$. We infer that 
$$
\sup_{t\in [0,T]} \|u^N(t)\|_{3+\alpha}\leq C K_{N,\alpha}^2.
$$
Finally, for $l\in \N^d$ with $|l|=3$, one can prove in the same way that $w_l:=D^l u^N$ satisfies
$$
\begin{array}{rl}
\ds \sup_{t\in [t_0,T]}\|w_l(t)\|_{1+\alpha}\; \leq & \ds C \left[ \|D^3_{xxx}F^N\|_\infty
+\|D^2_{xp}H\|_\infty\|D^3u^N\|_\infty+\|D^2_{pp}H\|_\infty\|D^2u^N\|_\infty\|D^3u^N\|_\infty
\right.\\
& \; \ds \left.+\|D^3_{xxp}H\|_\infty \|D^2u^N\|_\infty+\|D^3_{xpp}H\|_\infty\|D^2u^N\|_\infty^2+ \|D^2_{xp}H\|_\infty\|D^3u^N\|_\infty\right.\\
& \; \ds \left. +\|D^3_{xpp}H\|_\infty\|D^2u^N\|_\infty^2+ \|D^3_{ppp}H\|_\infty\|D^2u^N\|_\infty^3+\|D_{pp}^2H\|_\infty\|D^2u^N\|_\infty\|D^3u^N\|_\infty
\right.\\
& \; \ds \left.+\|D^3_{xxx}H\|_\infty+ \|D^3_{xxp}H\|_\infty\|D^2u^N\|_\infty\right]\; \leq \; C  K_{N,\alpha}^3.
\end{array}
$$
Therefore 
$$
\sup_{t\in [t_0,T]}\|u^N\|_{4+\alpha}\leq CK_{N,\alpha}^3.
$$
\end{proof}

%%%%%%%%%%%%%%%%%%%%
\subsection{Estimates for a linearized system}

We consider systems of the form 
\be\label{eq:prem}
\left\{ \begin{array}{rl}
 (i) & \ds - \partial_t z - \Delta z + V(t,x)\cdot Dz  = \frac{\delta F^N}{\delta m}(x, m^N(t))(\rho(t))+ b(t,x) \qquad {\rm in}\; [t_0,T]\times \T^d,\\
 (ii) & \ds \partial_t \rho - \Delta \rho - \dive(\rho V(t,x))- \dive (m^N\Gamma Dz+c)=0 \qquad  {\rm in}\; [t_0,T]\times \T^d,\\
 (ii) & \ds z(T,x)=0, \; \rho(t_0)=\rho_0\qquad  {\rm in}\;  \T^d,
\end{array}\right.
\ee
where  $V:[t_0,T]\times \R^d\to \R^d$ is a given vector field, $m^N\in C^0([0,T], \Pk)$, $\Gamma: [0,T]\times \T^d\to \R^{d\times d}$ is a continuous map with values into the family of symmetric matrices and where  the maps $b:[t_0,T]\times \T^d\to \R$ and
$c : [t_{0},T] \times \T^d \to \R^d$ are given. We  assume that, for any $\alpha\in (0,1)$, there is a constant $\bar C>0$ (depending on $\alpha$) such that
\be
\label{condGamma}
\begin{array}{c}
%\forall t,t'\in [t_0,T], \qquad \ds \dk(m^N(t),m^N(t'))\leq \bar C |t-t'|^{1/2}, (UTILE??)
%\\
\|D^kV\|_\infty\leq \bar CK_{N,\alpha}^{3k}\qquad \forall k\in\{0, \dots, 3\},\\
\ds \forall (t,x)\in [t_0,T]\times \T^d, \qquad 0\leq \Gamma(t,x)\leq  \bar CI_d.
\end{array}
\ee

Typically, $V(t,x)= D_pH(x,Du^N(t,x))$, $\Gamma(t,x)= D^2_{pp}H(x,Du^N(t,x))$ for some solution $(u^N,m^N)$ of the MFG system \eqref{e.MFG.t0m0} starting from some initial data $m(t_0)=m_0$. Proposition \ref{prop:reguum} then implies that \eqref{condGamma} holds. 

Following \cite{CDLL}, given $\rho_0\in C^{-(k+\alpha)}(\T^d)$, $(b,c)\in C^0([t_0,T], C^{k+\alpha}(\T^d)\times C^{-(k-1+\alpha)}(\T^d))$ (where $k\in \{1,\dots, 4\}$),  there exists a unique solution $(z,\rho)$  to system \eqref{eq:prem} in the sense of distribution in $C^0([0,T], C^{k+\alpha}(\T^d)\times C^{-(k+\alpha)}(\T^d))$.  

\begin{Proposition}\label{prop:estiLS} For any $k\in\{1,\dots,4\}$ and any $\alpha\in (0,1)$, we have,  
$$
\sup_{t\in [0,T]} \|z(t)\|_{k+\alpha} \; \leq \; C K_{N,\alpha}^{3k-2}M_{k}\qquad {\rm and}\qquad 
\sup_{t\in [t_0,T]} \|\rho(t)\|_{-(k+\alpha)} \; \leq C  K_{N,\alpha}^{3(k-1)}M_k,
$$
%$$
%\sup_t \| z\|_{k+\alpha}\; \leq \; C  K_{N,\alpha}^{8k-6}M_k\qquad {\rm and}\qquad 
%\sup_t \|\rho(t)\|_{-(k+\alpha)} \; \leq C  K_{N,\alpha}^{4k-3}M_k,
%$$
where $C$ depends on $\bar C$ and $\alpha$, but not on $N$, $V$, $\Gamma$, $m^N$, and where 
$$
M_k := \|\rho_0\|_{-(k+\alpha)} + \sup_{t\in [t_0,T]} \|c(t)\|_{-(k-1+\alpha)}+ \sup_{t\in [t_0,T]} \|b(t)\|_{k+\alpha}.
$$
\end{Proposition}

\begin{proof} To simplify the notation, we argue as if the solution $(z,\rho)$ is smooth. The general estimate is obtained by approximation (see \cite{CDLL}). 

{\it Step 1: Structure estimate. } Computing $\frac{d}{dt}\inte z\rho$, we find: 
$$
\frac{d}{dt}\inte z\rho = - \inte  \left\{\left(\frac{\delta F^N}{\delta m}(x, m^N)(\rho)+b\right)\rho + (m^N\Gamma Dz+c)\cdot Dz \right\}.
$$
So, using the initial and terminal conditions for $z$ and $\rho$ and the fact that $F^N$ is monotone, we get
\be\label{eq:energyineqLS}
\int_{t_0}^T \inte m^N \Gamma Dz\cdot Dz \leq \inte z(t_0)\rho(t_0)-\int_{t_0}^T \inte \left\{b\rho +c\cdot Dz\right\} .
\ee

\noindent {\it Step 2: Estimate for a linear backward equation.}
In order to estimate $\rho$, we use a duality method requiring estimates on a backward system: given $t_1\in (t_0,T]$ and $w^1\in C^\infty$, let $w$ solve 
\be\label{eq.wdual}
\left\{\begin{array}{l}
\ds - \partial_t w - \Delta w + V(t,x)\cdot Dw =0 \qquad {\rm in}\; [t_0,t_1]\times \T^d,\\
\ds w(t_1,x)=w^1(x)\qquad {\rm in}\; \T^d.
\end{array}\right.
\ee
where $w^1\in C^{\infty}(\T^d)$. Proposition \ref{prop.reguregu} in the Appendix states that there exists a constant $C$ depending on $\|V\|_\infty$, $d$, $\alpha$ only such that
$$
\sup_{t\in [0,T]}\|w(t)\|_{1+\alpha}\leq C  \|w^1\|_{1+\alpha}.
$$

Note that, for any $l\in \N^d$ with $k:=|l|\in\{1,\dots, 3\}$, the map $\hat w:= D^l w$ solves an equation of the form: 
$$
\left\{\begin{array}{l}
\ds - \partial_t \hat w - \Delta \hat w + V(t,x)\cdot D\hat w =g_l \qquad {\rm in}\; [t_0,t_1]\times \T^d,\\
\ds w(t_1,x)=D^lw^1(x)\qquad {\rm in}\; \T^d,
\end{array}\right.
$$
where $g_l$ is a linear combination of the $D^{l'}w$ with $1\leq |l'|\leq k$ and where the coefficient in front of $D^{l'}w$  is proportional to a derivative of order $k-|l'|+1$ of $V$ in the space variable. By Proposition \ref{prop.reguregu} and \eqref{condGamma} we get therefore 
$$
\begin{array}{rl}
\ds \sup_{t\in [0,T]} \|D^lw(t)\|_{1+\alpha} \; \leq & \ds C \left[ \|D^lw^1\|_{1+\alpha}+\|g_l\|_\infty\right]\\
\leq & \ds C \left[ \|D^lw^1\|_{1+\alpha}+ \sum_{1\leq |l'|\leq k} K^{3(k-|l'|+1)}_{N,\alpha} \|D^{l'}w\|_\infty\right].
\end{array}
$$
By induction, this implies that, for $k\in \{1, \dots,4\}$,  
\be\label{estiw}
\sup_{t\in [0,T]} \|w(t)\|_{k+\alpha}\leq C K^{3(k-1)}_{N,\alpha}\|w^1\|_{k+\alpha}.
\ee

{\it Step 3: Estimate of $\rho$ by duality.} Let us fix $t_1\in (t_0,T]$, $w^1\in C^\infty$ and let $w$ be the solution to \eqref{eq.wdual}. As $\rho$ solves \eqref{eq:prem}, we have, for $k\in\{1,\dots ,4\}$,  
$$
\begin{array}{l}
\ds \inte w^1\rho(t_1)\; = \; \ds \inte w(t_0)\rho_0 -\int_{t_0}^{t_1}\inte \left(m^N \Gamma Dz+c\right)\cdot Dw \\
 \qquad \leq  \ds \|\rho_0\|_{-(k+\alpha)}\|w(t_0)\|_{k+\alpha}+ \bigl(\int_{t_0}^T \inte m^N\Gamma Dz\cdot Dz\bigr)^{\frac12}\bigl( \int_{t_0}^{t_1} \inte m^N\Gamma Dw\cdot Dw\bigr)^{\frac12}\\
\ds \qquad  \qquad \qquad \qquad \qquad \qquad \qquad + C\sup_t \|c(t)\|_{-(k-1+\alpha)}\sup_t \|Dw\|_{k-1+\alpha}\\
 \qquad \leq  \ds  \|\rho_0\|_{-(k+\alpha)}\|w(t_0)\|_{k+\alpha}+  C\|Dw\|_\infty \bigl(\int_{t_0}^T \inte m^N\Gamma Dz\cdot Dz\bigr)^{\frac12}\\
\ds \qquad  \qquad \qquad \qquad \qquad \qquad \qquad + C\sup_t \|c(t)\|_{-(k-1+\alpha)}\sup_t \|Dw\|_{k-1+\alpha},
\end{array}
$$
where we used the fact that $\Gamma$ is bounded and $\inte m^N(t)=1$ in the last inequality. 
Recalling \eqref{estiw}, we get:
$$
\begin{array}{rl}
\ds \inte w^1\rho(t_1)\;   \leq  &  \ds C\|w^1\|_{k+\alpha}\left\{ K_{N,\alpha}^{3(k-1)}\bigl(\|\rho_0\|_{-(k+\alpha)} + \sup_t \|c(t)\|_{-(k-1+\alpha)}\bigr)+ \bigl(\int_{t_0}^T \inte m^N\Gamma Dz\cdot Dz\bigr)^{\frac12}\right\}.
\end{array}
$$ 
Taking the supremum with respect to $t_1$  and to $w^1$ with $\|w^1\|_{k+\alpha}\leq 1$, we obtain therefore
$$
\sup_t \|\rho(t)\|_{-(k+\alpha)} \leq CK_{N,\alpha}^{3(k-1)}\left( \|\rho_0\|_{-(k+\alpha)} + \sup_t \|c(t)\|_{-(k-1+\alpha)}\right)+ C\left(\int_{t_0}^T \inte m^N\Gamma Dz\cdot Dz\right)^{\frac12}.
$$
For $r\geq 1$, we plug \eqref{eq:energyineqLS} into the above estimate:
$$
\begin{array}{rl}
\ds \sup_t \|\rho(t)\|_{-(k+\alpha)} \; \leq & \ds C  K_{N,\alpha}^{3(k-1)}\left(\|\rho_0\|_{-(k+\alpha)} + \sup_t \|c(t)\|_{-(k-1+\alpha)}\right)+
C\|z(t_0)\|_{r}^{\frac12}\|\rho_0\|_{-r}^{\frac12} \\
& \ds \; +C\left(\sup_t \|b(t)\|_{k+\alpha}^\frac12 \sup_t\|\rho(t)\|_{-(k+\alpha)}^\frac12 +
\sup_t\|c(t)\|_{-(r-1)}^\frac12\sup_t\|z(t)\|_{r}^\frac12\right).
\end{array}
$$
Rearranging we find: 
$$
\sup_{t\in [t_0,T]} \|\rho(t)\|_{-(k+\alpha)} \; \leq C  K_{N,\alpha}^{3(k-1)}M_k+C\sup_{t\in [t_0,T]} \|z(t)\|_{r}^{\frac12}\left(\|\rho_0\|_{-r}^{\frac12}+\sup_{t\in [t_0,T]}\|c(t)\|_{-(r-1)}^\frac12\right),
$$
where $M_k$ is defined in the Proposition. \\

\noindent {\it Step 4: Estimate of $z$.} Fix $l\in \N^d$ with $k:=|l|\in \{0, \cdots, 4\}$. In view of the equation satisfied by $z$, the map $\hat z:= D^l z$ solves an equation of the form 
$$
\left\{ \begin{array}{l}
- \partial_t \hat z - \Delta \hat z + V\cdot D\hat z  = D^l\frac{\delta F^N}{\delta m}(x, m^N(t))(\rho(t))+ D^lb
+g_l \qquad {\rm in}\; [t_0,T]\times \T^d,\\
 \ds \hat z(T,x)=0\qquad  {\rm in}\;  \T^d,
\end{array}\right.
$$
where $g_l$ is as in step 2 with $z$ replacing $w$. Proposition \ref{prop.reguregu} thus implies that 
$$
\sup_{t\in [0,T]} \|D^l z(t)\|_{1+\alpha} \leq 
C\left[  \|D^l\frac{\delta F^N}{\delta m}(\cdot, m^N(\cdot))(\rho(\cdot))\|_\infty+ \|D^lb\|_\infty+ 
\sum_{1\leq |l'|\leq k} K^{3(k-|l'|+1)}_{N,\alpha} \|D^{l'}z\|_\infty\right],
$$
where, from assumption \eqref{defKep1} and Step 3 for $r=k+1+\alpha$: 
$$
\begin{array}{l}
\ds \left\|D^l\frac{\delta F^N}{\delta m}(\cdot, m^N(\cdot))(\rho(\cdot))\right\|_\infty \;  \leq \; \ds 
\sup_m\left\|\frac{\delta F^N}{\delta m}(\cdot, m,\cdot)\right\|_{k+\alpha,k+1+\alpha}\sup_t  \|\rho(t)\|_{-(k+1+\alpha)} \\
\qquad \qquad \leq \; \ds C  K_{N,\alpha}^{3k+1}M_{k+1}+CK_{N,\alpha}\sup_t \|z(t)\|_{k+1+\alpha}^{\frac12}\left(\|\rho_0\|_{-(k+1+\alpha)}^{\frac12}+\sup_t\|c(t)\|_{-(k+\alpha)}^\frac12\right).
\end{array}
$$
So we find
$$
\begin{array}{rl}
\ds \sup_{t\in [0,T]} \|z(t)\|_{k+1+\alpha} \; \leq  & \ds
C\left[  K_{N,\alpha}^{3k+1}M_{k+1}+K_{N,\alpha}\sup_t \|z(t)\|_{k+1+\alpha}^{\frac12}\left(\|\rho_0\|_{-(k+1+\alpha)}^{\frac12}+\sup_t\|c(t)\|_{-(k+\alpha)}^\frac12\right) \right.\\
&\ds  \left.\qquad + \|b\|_k+ \sum_{1\leq n \leq k} K^{3(k-n+1)}_{N,\alpha}\sup_{t\in [0,T]} \|z(t)\|_{n+\alpha}\right].
\end{array}
$$
We rearrange the expression to obtain
$$
\begin{array}{rl}
\ds \sup_{t\in [0,T]} \|z(t)\|_{k+1+\alpha} \; \leq  & \ds
C\left[  K_{N,\alpha}^{3k+1}M_{k+1}+K_{N,\alpha}^2\left(\|\rho_0\|_{-(k+1+\alpha)}+\sup_t\|c(t)\|_{-(k+\alpha)}\right) \right.\\
&\ds  \left.\qquad + \|b\|_k+ \sum_{1\leq n \leq k} K^{3(k-n+1)}_{N,\alpha} \sup_{t\in [0,T]} \|z(t)\|_{n+\alpha}\right]\\
\leq & \ds C\left[  K_{N,\alpha}^{3k+1}M_{k+1}+ \sum_{1\leq n \leq k} K^{3(k-n+1)}_{N,\alpha} \sup_{t\in [0,T]} \|z(t)\|_{n+\alpha}\right].
\end{array}
$$
By induction we infer that, for $k\in \{1, \dots, 4\}$,  
$$
\sup_{t\in [0,T]} \|z(t)\|_{k+\alpha} \; \leq \; C K_{N,\alpha}^{3k-2}M_{k}.
$$
Plugging this inequality into our estimate for $\rho$ (in step 3) gives: 
$$
\sup_{t\in [t_0,T]} \|\rho(t)\|_{-(k+\alpha)} \; \leq C  K_{N,\alpha}^{3(k-1)}M_k.
$$
\end{proof}

%%%%%%%%%%%%%%%%%%%%%%%%%%%%%%%%%%%%%
\subsection{Estimates for $\frac{\delta U^N}{\delta m}$}
\label{subse:partie:1:differentiability}

In this section we provide estimates for $\frac{\delta U^N}{\delta m}$ where $U^N$ is  the solution of the master equation \eqref{MasterEq}. Following the construction of \cite{CDLL}, we can express this derivative in terms of a linearized system. Let us fix $(t_0,m_0)\in [0,T]\times \Pk$ and let $(m^N,u^N)$ be the solution to the MFG system \eqref{e.MFG.t0m0} with initial condition $m(t_0)=m_0$. Recall that, by  definition, $U^N(t_0,x,m_0)=u^N(t_0,x)$. 

For any $\mu_0\in C^\infty(\T^d)$, we consider the solution $(z,\rho)$ to the linearized system
\be\label{eq:vmubis}
\left\{ \begin{array}{l}
\ds - \partial_t z - \Delta z + D_pH(x,Du^N)\cdot Dz =\frac{\delta F^N}{\delta m}(x,m^N(t))(\rho(t)) \; {\rm in}\; (0,T)\times \T^d,\\
\ds \partial_t \rho - \Delta \rho -{\rm div}( \rho  D_pH(x, D u^N))-{\rm div}( m^N  D^2_{pp}H(x, D u^N)Dz)=0  \; {\rm in}\; (0,T)\times \T^d,\\
\ds z(T,\cdot)=0, \; \rho(t_0,\cdot)=\rho_0  \; {\rm in}\; \T^d.
\end{array}\right.
\ee
We proved in \cite{CDLL} the identity
$$
z(t_0,x)= \inte \frac{\delta U^N}{\delta m}(t_0,x,m_0,y) \rho_0(y)dy.
$$
In order to estimate $\frac{\delta U^N}{\delta m}$, which just need to estimate $z$: this is the aim of the next statement.

\begin{Proposition}\label{prop:estiDmU} The unique solution  $(z,\rho)$ of \eqref{eq:vmubis} satisfies, for $k\in\{1,\dots,4\}$ and any $\alpha\in (0,1)$, 
\be\label{eq:estizk}
\ds \sup_{t\in [t_0,T]}  \|z(t,\cdot)\|_{k+\alpha}\;  \leq \; \ds  CK_{N,\alpha}^{3k-2}\|\rho_0\|_{-(k+\alpha)}, 
\ee
\be\label{eq:estirhok}
\ds \sup_{t\in [t_0,T]}  \|\rho(t)\|_{-(k+\alpha)}\;  \leq \; \ds  CK_{N,\alpha}^{3(k-1)}\|\rho_0\|_{-(k+\alpha)}, 
\ee
where the  constant $C$ does not depend on $(t_0,m_0)$ nor on $N$. 
\end{Proposition}

\begin{proof} It is a straightforward application of Proposition \ref{prop:estiLS}, with $V(t,x)= D_pH(x,Du^N(t,x))$, $\Gamma(t,x)=D^2_{pp}H(x,Du^N(t,x))$ and $b=c=0$.  
\end{proof}

As in \cite{CDLL} we can derive from the Proposition an estimate on the first order derivative of $U$ with respect to the measure: 

\begin{Corollary}\label{cor:estiDmU} For any $(t_0,m_0)\in [0,T]\times \Pk$ and $k\in\{1,\dots,4\}$, we have: 
$$
 \left\| \frac{\delta U^N}{\delta m}(t_0,\cdot,m_0,\cdot)\right\|_{(k+\alpha, k+\alpha)} +
 \left\| D_m U^N(t_0,\cdot,m_0,\cdot)\right\|_{(k+\alpha, k-1+\alpha)}  \leq  CK_{N,\alpha}^{3k-2}.
 $$
\end{Corollary}

%%%%%%%%%%%%%%%%%%%%%%%%%%%%
\subsection{Estimate for $\frac{\delta^2 U^N}{\delta m^2}$}

We now estimate the second order derivative with respect to $m$ of the solution $U^N$ to the master equation \eqref{MasterEq}. 
Let us fix $(t_0,m_0)\in [0,T]\times \Pk$ and let $(m^N,u^N)$ be the solution to the MFG system \eqref{e.MFG.t0m0} with initial condition $m(t_0)=m_0$. Let $(z,\rho)$ be a solution of the linearized system \eqref{eq:vmubis} with initial condition $\rho_0$. The second order linearized system  reads 
\be\label{eq:LSordre2}
\left\{ \begin{array}{l}
\ds - \partial_t w - \Delta w + D_pH(x,Du^N)\cdot Dw =\frac{\delta F^N}{\delta m}(x,m^N(t))(\mu(t))\\
\ds \qquad  + \frac{\delta^2 F^N}{\delta m^2}(x,m^N(t))(\rho(t),\rho(t))- D^2_{pp}H(x,Du^N)Dz\cdot Dz\; {\rm in}\; (0,T)\times \T^d,\\
\;\\
\ds \partial_t \mu - \Delta \mu -{\rm div}( \mu  D_pH(x, D u^N))-{\rm div}( m^N  D^2_{pp}H(x, D u^N)Dw)\\
\ds \qquad   =\dive\left(m^N D^3_{ppp}H(x,Du^N)DzDz\right)+2\dive \left(\rho D^2_{pp}H(x,Du^N)Dz\right) \; {\rm in}\; (0,T)\times \T^d,\\
\ds w(T,\cdot)=0, \; \mu(t_0,\cdot)=0\; {\rm in}\;  \T^d.
\end{array}\right.
\ee
Following \cite{CDLL}, we have 
$$
w(t_0,x)= \inte\inte \frac{\delta^2 U^N}{\delta m^2}(t_0,x,m_0,y,y')\rho_0(y)\rho_0(y')dydy'.
$$

 \begin{Proposition}\label{prop:estiLS2} We have, for $k=2,3$, 
$$
\sup_{t\in [t_0,T]} \|w(t)\|_{k+\alpha}\leq C K_{N,\alpha}^{12k}\|\rho_0\|_{-(k-1+\alpha)}^2.
$$
\end{Proposition}

As a consequence, we have: 
\begin{Corollary}\label{cor:delta2U} For any $(t_0,m_0)\in [0,T]\times \Pk$ and $k=2,3$, we have: 
$$
\left\| \frac{\delta^2 U^N}{\delta m^2}(t_0,\cdot,m_0,\cdot,\cdot)\right\|_{k+\alpha,k-1+\alpha,k-1+\alpha}
+\left\| D^2_{mm} U^N(t_0,\cdot,m_0,\cdot,\cdot)\right\|_{k+\alpha,k-2+ \alpha,k-2+\alpha}
\leq C K_{N,\alpha}^{12k}.
$$
%and
%$$
%\left\| D^2_{mm} U^N(t_0,\cdot,m_0,\cdot,\cdot)\right\|_{3+\beta,\beta,\beta}\leq C K_{N,\alpha}^{15}.
%$$
\end{Corollary}

\begin{proof}[Proof of Proposition \ref{prop:estiLS2}.] We apply Proposition \ref{prop:estiLS} to $(w,\mu)$ with initial condition $\mu(t_0)=0$ and
$$
b(t)=  \frac{\delta^2 F^N}{\delta m^2}(x,m^N(t))(\rho(t),\rho(t))- D^2_{pp}H(x,Du^N)Dz\cdot Dz,
$$
$$
c(t)= \left(m^N D^3_{ppp}H(x,Du^N)DzDz\right)+2 \rho D^2_{pp}H(x,Du^N)Dz.
$$
We have, for any $k\in \{1, \dots, 3\}$,  
$$
\begin{array}{rl}
\ds \sup_t \|b(t)\|_{k+\alpha}\; \leq & \ds CK_{N,\alpha}\sup_t \|\rho(t)\|^2_{-(k-1+\alpha)}+ C\sup_t \|D^2_{pp}H(\cdot,Du^N(t))\|_{k+\alpha}\sup_t \|z(t)\|^2_{k+1+\alpha} \\
%\leq & \ds C K_{N,\alpha}^{56} \|\rho_0\|_{-(4+\alpha)}
\leq & \ds C K_{N,\alpha}^{1+6(k-2)} \|\rho_0\|_{-(k-1+\alpha)}^2+ C K_{N,\alpha}^{3k+ 2(3(k+1)-2)}\|\rho_0\|_{-(k+1+\alpha)}^2\\
\leq & \ds C K_{N,\alpha}^{9k+2} \|\rho_0\|_{-(k-1+\alpha)}^2, 
\end{array}
$$
where we used Proposition \ref{prop:reguum}, \eqref{eq:estizk}  and \eqref{eq:estirhok}. Next we estimate $c$ for $k\in \{2,3\}$:
$$
\begin{array}{rl}
\ds \|c(t)\|_{-(k-1+\alpha)}\; \leq  & \ds C \sup_{\|\phi\|_{k-1+\alpha}\leq 1} \left(\inte m^N \phi  D^3_{ppp}H(x,Du^N)DzDz + \inte \phi \rho D^2_{pp}H(x,Du^N)Dz\right) \\
\leq & \ds C\|Dz\|_\infty^2 + C \sup_{\|\phi\|_{k-1+\alpha}\leq 1} \| \phi D^2_{pp}H(x,Du^N)Dz\|_{k-1+\alpha}\sup_t \| \rho(t)\|_{-(k-1+\alpha)}\\
\leq & \ds C K_{N,\alpha}^{3(k-1)-2}\|\rho_0\|_{-(k-1+\alpha)}^2+CK_{N,\alpha}^{3(k-1)+3k-2+3(k-2)}\|\rho_0\|_{-(k+\alpha)}\|\rho_0\|_{-(k-1+\alpha)} \\
\leq & \ds C  K_{N,\alpha}^{9(k-1)-2}\|\rho_0\|_{-(k-1+\alpha)}^2.
\end{array}
$$
Therefore, by Proposition \ref{prop:estiLS}, we obtain,  for $k\in \{2,3\}$:  
$$
\sup_t \|w(t)\|_{k+\alpha}\leq C K_{N,\alpha}^{3k-2}\left[\sup_t \|b(t)\|_{k+\alpha}+\sup_t   \|c(t)\|_{-(k-1+\alpha)}\right]\leq  C K_{N,\alpha}^{12k}\|\rho_0\|_{-(k-1+\alpha)}^2.
$$
\end{proof}

%%%%%%%%%%%%%%%%%%%%%%%%%%%%%%%%%%%%%
%%%%%%%%%%%%%%%%%%%%%%%%%%%%%%%%%%%%
\section{Convergence}\label{sec.convergence}

In this section, we consider,
for an integer $N \geq 2$, a classical solution $(v^{N,i})_{i \in \{1,\dots,N\}}$ of the Nash system:  
\be\label{Nash0}
\left\{ \begin{array}{l}
\ds - \partial_t v^{N,i}(t,\bx) -  \sum_{j} \Delta_{x_j}v^{N,i}(t,\bx)+ H\bigl(x_i,  D_{x_i}v^{N,i}(t,\bx)
\bigr) \\
\ds \qquad  + \sum_{j\neq i}  D_pH \bigl(x_j, D_{x_j}v^{N,j}(t,\bx)\bigr)
\cdot  D_{x_j}v^{N,i}(t,\bx)=
F^N(x_i, m^{N,i}_{\bx})\; {\rm in }\; [0,T]\times (\T^{d})^N,
\\
\ds v^{N,i}(T,\bx)= G(x_i)\qquad {\rm in }\;  (\T^{d})^N,
\end{array}\right.
\ee
where we have set, for $\ds {\bx}=(x_1, \dots, x_N)\in (\T^{d})^N$, $\ds m^{N,i}_{\bx}=\frac{1}{N-1}\sum_{j\neq i} \delta_{x_j}$. 

As $H$ is Lipschitz continuous, system \eqref{Nash0} has a unique classical solution \cite{LSU}. By uniqueness, the $v^{N,i}$ enjoy strong symmetry properties. On the one hand, $v^{N,i}(t,x_1, \dots, x_N)$ is symmetric with respect to the variables $(x_j)_{j\neq i}$. On the other hand, for $j\neq i$, $v^{N,i}(t,\bx)=v^{N,j}(t,\by)$, where $\bx=(x_1, \dots, x_N)$ and $\by$ is obtained from $\bx$ by permuting the $x_i$ and $x_j$ variables. 

Our aim is to quantify the convergence rate of $v^{N,i}$ to the solution $U^N$ of the master equation \eqref{MasterEq} as $N$ tends to $\infty$.

%%%%%%%%%%%%%%%%%%%%%%%%%%%%%%%%%%%
\subsection{Finite dimensional projections of $U^N$}

Let $U^N=U^N(t,x,m)$ be the solution of the second order master equation \eqref{MasterEq}. 
For $N\geq 2$ and $i\in \{1,\dots, N\}$ we set 
$$
u^{N,i}(t,{\bx})= U^N(t,x_i, m^{N,i}_{\bx})\quad {\rm where }\; {\bx}=(x_1, \dots, x_N)\in (\T^{d})^N, \; m^{N,i}_{\bx}=\frac{1}{N-1}\sum_{j\neq i} \delta_{x_j}.
$$
Following \cite{CDLL}, we know that the $u^{N,i}$ are of class $C^2$ with respect to the space variables and $C^1$ with respect to the time variable, with, for $i$, $j$, $k$ distinct: 
\be\label{uNC2}
\begin{array}{c}
\ds \partial_t u^{N,i}(t,{\bx})= \partial_t U^N(t,x_i,m^{N,i}_{\bx}), \qquad D_{x_i} u^{N,i}(t,{\bx})= D_xU^N(t,x_i,m^{N,i}_{\bx}) ,\\
\ds D^2_{x_ix_i} u^{N,i}(t,{\bx})= D^2_{xx} U^N(t,x_i,m^{N,i}_{\bx}),  \qquad
\ds D_{x_j}  u^{N,i}(t,{\bx})= \frac{1}{N-1} D_mU^N(t,x_i,m^{N,i}_{\bx},x_j) ,\\
\ds D^2_{x_jx_k} u^{N,i}(t,{\bx})= \frac{1}{(N-1)^2} D^2_{mm}U^N(t,x_i,m^{N,i}_{\bx},x_j,x_k),\\
\ds D^2_{x_jx_j} u^{N,i}(t,{\bx})= \frac{1}{(N-1)^2} D^2_{mm}U^N(t,x_i,m^{N,i}_{\bx},x_j,x_j) +
 \frac{1}{N-1} D^2_{ym}U^N(t,x_i,m^{N,i}_{\bx},x_j).
\end{array}
\ee

We estimate how far  $(u^{N,i})_{i \in \{1,\dots,N\}}$ is to be a solution to the Nash system \eqref{Nash0}:   

\begin{Proposition}\label{Prop:equNi} The map $(v^{N,i})$ satisfies
\be\label{eq:uNi}
\left\{\begin{array}{l}
\ds  - \partial_t u^{N,i}   - \sum_j \Delta_{x_j} u^{N,i}  +H(x_i,D_{x_i}u^{N,i}) \\
\qquad \ds  +\sum_{j\neq i}  D_{x_j}u^{N,i}(t,{\bx})\cdot  D_pH 
\bigl(x_j,D_{x_j}u^{N,j}(t,{\bx}) \bigr) 
=F^N(x_i,m^{N,i}_{\bx}) +r^{N,i}(t,{\bx}) 
\\
\qquad \qquad \qquad \qquad \qquad  \ds  \qquad 
\hspace{4pt}
\mbox{\rm a.e. in}\; (0,T)\times \T^{Nd},
\vspace{2pt}
\\
u^{N,i}(T,{\bx})= G(x_i) \qquad  {\rm in}\;  \T^{Nd},
\end{array}\right.
 \ee
where $r^{N,i}\in C^0([0,T]\times \T^d)$ with
$$
\|r^{N,i}\|_\infty\leq \frac{C}{N}\left(\|D_mU^N\|_\infty \|D_{m,x}U^N\|_{\infty}+\|D^2_{mm}U^N\|_\infty\right).
$$
\end{Proposition}

\begin{proof}  
%Recalling \eqref{uNC2} and Proposition \ref{prop:DmULip}, we note that the derivatives $D_{x_j}u^{N,i}$ (for $j\neq i$) are Lipschitz continuous in the space variables. In particular, the $u^{N,i}$ are in $W^{2,\infty}$ with respect to the space variables, with 
%\be\label{kjhbKS}
%\left|D^2_{x_j,x_j} u^{N,i}(t,\bx) - \frac{1}{N-1}D_{y} D_mU^N(t,x_i,m^{N,i}_{\bx}, x_j)\right|\leq \frac{CK_{N,\alpha}^4}{N^2} \qquad {\rm a.e.}.
%\ee
As $U^N$ solves \eqref{MasterEq}, one has at a point $(t,x_i,m^{N,i}_{\bx})$: 
\begin{align*} 
&- \partial_t U^N  - \Delta_x U^N +H(x_i,D_xU^N)  -\inte \dive_y \bigl[D_m U^N\bigr]\bigl(t,x_i,m^{N,i}_{\bx},y \bigr) d m^{N,i}_{\bx}(y) 
\\
&\hspace{15pt}+ \inte  D_m U^N\bigl(t,x_i,m^{N,i}_{\bx},y\bigr)\cdot  
D_pH\bigl(y,D_xU^N(t,y,m^{N,i}_{\bx}) \bigr)
dm^{N,i}_{\bx}(y) =F^N\bigl(x_i,m^{N,i}_{\bx}\bigr). 
\end{align*}
So $u^{N,i}$ satisfies:  
\begin{align*}
&- \partial_t u^{N,i}  - \Delta_{x_i} u^{N,i} +H(x_i,D_{x_i}u^{N,i})  -  \inte \dive_y\bigl[ D_m U^N\bigr]\bigl(t,x_i,m^{N,i}_{\bx},y\bigr) d m^{N,i}_{\bx}(y)
  \\
&\hspace{30pt} +\frac{1}{N-1}\sum_{j\neq i}  D_m U^N\bigl(t,x_i,m^{N,i}_{\bx},x_j\bigr) \cdot  D_pH \bigl(x_j,D_xU^N(t,x_j,m^{N,i}_{\bx})\bigr)
=F^N(x_i,m^{N,i}_{\bx}) .
\end{align*}
 Note that, by \eqref{uNC2}, for any $j\neq i$, we have:
 $$
\frac{1}{N-1} D_m U^N \bigl(t,x_i,m^{N,i}_{\bx},x_j\bigr) = D_{x_j}u^{N,i}(t,\bx).
$$
In particular,  
\be\label{e.boundDxjui}
\| D_{x_j}u^{N,i}\|_\infty\leq \frac{1}{N}\|D_mU^N\|_\infty.
\ee
By the Lipschitz continuity of $D_xU^N$ with respect to $m$, we have 
$$
\left|D_xU^N(t,x_j,m^{N,i}_{\bx})- D_xU^N(t,x_j,m^{N,j}_{\bx})\right| \leq \|D_{m,x}U^N\|_\infty \dk(m^{N,i}_{\bx}, m^{N,j}_{\bx})\leq \frac{C}{N}\|D_{m,x}U^N\|_\infty, 
$$
so that,
%$$
%\left|D_xU(t,x_j,m^{N,i}_{\bx})- D_{x_j}u^{N,j}(t,{\bx})\right| \leq \frac{C}{N}
%$$
%and, 
by Lipschitz continuity of $D_pH$, 
\begin{equation}
\label{eq:quelles:H:utiles}
\bigl| D_pH \bigl(x_j,D_xU^N(t,x_j,m^{N,i}_{\bx})\bigr)- D_pH 
\bigl(x_j,D_{x_j}u^{N,j}(t,{\bx})\bigr)\bigr|\leq  \frac{C}{N}\|D_{m,x}U^N\|_\infty.
\end{equation}
Therefore
\begin{align*}
&\frac{1}{N-1}\sum_{j\neq i}  D_m U^N\bigl(t,x_i,m^{N,i}_{\bx},x_j\bigr) \cdot  D_pH \bigl(x_j,D_xU^N(t,x_j,m^{N,i}_{\bx})\bigr) 
\\
%&\hspace{15pt} = \sum_{j\neq i}  D_{x_j}u^{N,i}(t,{\bx})\cdot  D_pH \bigl(x_j,D_xU(t,x_j,m^{N,i}_{\bx})\bigr) 
%\\
&\hspace{15pt} =
\sum_{j\neq i}  D_{x_j}u^{N,i}(t,{\bx})\cdot  D_pH \bigl(x_j,D_{x_j}u^{N,j}(t,{\bx})\bigr) +r^{N,i}_1(t,\bx)=0,
\end{align*}
where, by \eqref{e.boundDxjui} and \eqref{eq:quelles:H:utiles},  
$$
\left\|r^{N,i}_1\right\|_\infty\leq \frac{C}{N}  \|D_mU^N\|_\infty \|D_{m,x}U^N\|_\infty.
$$
On the other hand, by \eqref{uNC2}, we have 
$$
\begin{array}{l}
\ds \sum_{j=1}^N \Delta_{x_j} u^{N,i}   - \Delta_xU^N(t,x_i, m^{N,i}_{\bx})- \frac{1}{N-1}\sum_{j\neq i} \dive_y D_mU^N(t,x_i,m^{N,i}_{\bx}, x_j) \\
\ds \qquad = \frac{1}{(N-1)^2} \sum_{j\neq i} {\rm tr}(D^2_{mm}U^N(t,x_i,m^{N,i}_{\bx}, x_j,x_j)) =:r^{N,i}_2(t,\bx),
\end{array}$$
where 
$$
\|r^{N,i}_2\|_\infty \leq \frac{C}{N}\|D^2_{mm}U^N\|_\infty.
$$
Therefore 
\begin{align*}
&- \partial_t u^{N,i}(t,\bx) - \sum_j \Delta_{x_j} u^{N,i}(t,\bx) 
  + H\bigl(x_i,D_{x_i}u^{N,i}(t,\bx) \bigr)    
  \\
&\hspace{15pt} + \sum_{j\neq i}  D_{x_j}u^{N,i}(t,{\bx})\cdot  D_pH 
\bigl(x_j,D_{x_j}u^{N,j}(t,{\bx})\bigr)  =F(x_i,m^{N,i}_{\bx})
+ 
r^{N,i}_1+r^{N,i}_2,
\end{align*}
which shows the result.
\end{proof}

%%%%%%%%%%%%%%%%%%%%%%%%%%%%%%%%%%%%
%%%%%%%%%%%%%%%%%%%%%%%%%%%%%%%%%%%%%
\subsection{Estimates between $v^{N,i}$ and $U^N$}

Let us fix $t_0\in [0,T)$. Let $(Z_i)_{i \in \{1,\dots,N\}}$ be an i.i.d family of $N$ random variables. We set $\bZ=(Z_i)_{i \in \{1,\dots,N\}}$.  Let also $((B_{t}^{i})_{t \in [ 0,T]})_{i \in \{1,\dots,N\}}$ be a family of $N$ independent $d$-dimensional Brownian Motions which is also independent of $(Z_i)_{i \in \{1,\dots,N\}}$. 
We consider the systems of SDEs with variables $(\bX_{t}=(X_{i,t})_{i\in \{1,\dots,N\}})_{t \in [0,T]}$ and $(\bY_{t}=(Y_{i,t})_{i\in \{1,\dots,N\}})_{t \in [0,T]}$: 
\be\label{defxitBIS}
\left\{\begin{array}{l}
dX_{i,t}= -D_pH\bigl(X_{i,t}, D_{x_i}u^{N,i}(t, \bX_{t})\bigr){\rm dt} +\sqrt{2} dB^{i}_t\qquad t\in [t_0,T],\\
X_{i,t_0}= Z_i,
\end{array}\right.
\ee
and 
\be\label{defyitBIS}
\left\{\begin{array}{l}
dY_{i,t}= -D_pH\bigl(Y_{i,t}, D_{x_i}v^{N,i}(t, \bY_t)\bigr){\rm dt} +\sqrt{2} dB^{i}_t\qquad t\in [t_0,T],\\
Y_{i,t_0}= Z_i.
\end{array}\right.
\ee
Note that, by the symmetry properties of the $(u^{N,i})_{i \in \{1,\dots,N\}}$ and of the $(v^{N,i})_{i \in \{1,\dots,N\}}$, 
the processes $(X_{i,t}, Y_{i,t})_{t\in [t_0,T]})_{i \in \{1,\dots,N\}}$ are exchangeable.

Let us finally introduce notations for the error terms: 
\be\label{alphaNbetaN}
\alpha^N=\sup_i \|D_{x_i}u^{N,i}\|_\infty,
\qquad
\beta^N=\sup_i\sup_{j\neq i} \|D_{x_j}u^{N,i}\|_\infty
\qquad {\rm
and} \qquad
r^N:=\sup_{i}\|r^{N,i}\|_\infty,
\ee
where $r^{N,i}$ is the error term in Proposition \ref{Prop:equNi}. In the same way, we set
$$
\hat \alpha^N=\sup_i \|D^2_{x_i,x_i}u^{N,i}\|_\infty,
\qquad
\hat \beta^N=\sup_i\sup_{j\neq i} \|D^2_{x_i,x_j}u^{N,i}\|_\infty.
$$
Finally,  
\be\label{thetaN}
\theta^N:= \left(1+(\alpha^N)^2+(N\beta^N)^2\right), \qquad \hat \theta^N := (1+ \hat \alpha^N+N\hat \beta^N).
\ee
Note that, by symmetry, the $\sup_i$ in the above expressions is actually superfluous. Theorem \ref{them:KeyEsti} implies that $r^N$ and $\theta^N$ are of the following order: 
$$
r^N\leq \frac{CK_{N,\alpha}^{12}}{N}, \qquad \theta^N  \leq CK_{N,\alpha}^{6}, \qquad  \hat \theta^N \leq CK_{N,\alpha}^{3}.
$$

\begin{Theorem}\label{thm:CvyxBIS} We have, for any $i\in \{1, \cdots, N\}$,  
\be\label{uNi-vNi}
\|u^{N,i}-v^{N,i}\|_\infty\leq  Cr^N(\theta^N)^{1/2} \exp(C\theta^N),
\ee
\be\label{estixi-yiBIS}
\E\bigl[\sup_{t\in [t_0,T]} |Y_{i,t}-X_{i,t}|\bigr]\leq Cr^N(\theta^N\hat \theta^N)^{1/2}\exp(C(\theta^N+\hat \theta^N))
\ee
and 
\be\label{CvyxIneq2BIS}
\begin{array}{r}
\ds \E\biggl[\int_{t_0}^T  |D_{x_i}v^{N,i}(t, \bY_t)- D_{x_i}u^{N,i}(t,\bY_t)|^2{\rm dt} \biggr] 
% \qquad \\
%\qquad  
\leq  C(r^N)^2\theta^N \exp(C\theta^N),
\end{array}
\ee
where $C$ is a (deterministic) constant that does not depend on $t_0$, $m_0$ and $N$.
\end{Theorem}

\begin{proof} We follow closely the proof in \cite{CDLL} and so indicate only the main changes. 
We will use the following notations: for $t\in [t_0,T]$, 
\begin{equation*} 
\begin{split}
&U^{N,i}_{t} = u^{N,i}(t,\bY_{t}), \quad 
V^{N,i}_{t} = v^{N,i}(t,\bY_{t}), 
\\
&DU^{N,i,j}_{t} = D_{x_{j}} u^{N,i}(t,\bY_{t}), 
\quad
DV^{N,i,j}_{t} = D_{x_{j}} v^{N,i}(t,\bY_{t}).
\end{split}
\end{equation*}
As the $(v^{N,i})_{i \in \{1,\dots,N\}}$ solve equation \eqref{Nash0}, we have by It\^o's formula that, for any $i \in \{1,\dots,N\}$,
\begin{equation}
\label{rep:vNi}
\begin{split}
d V^{N,i}_{t}
&= 
\Bigl[ H \bigl( Y_{i,t},D_{x_i} v^{N,i}(t, \bY_t) \bigr)
- D_{x_{i}} v^{N,i}(t,\bY_{t}) 
\cdot
D_pH\bigl(Y_{i,t}, 
D_{x_i} v^{N,i}(t, \bY_t) \bigr)
\\
&\hspace{150pt}
- F^N \bigl( Y_{i,t},m^{N,i}_{\bY_t})
\Bigr] {\rm dt}
\\
&\hspace{15pt} + 
\sqrt{2}
\sum_{j} D_{x_{j}} v^{N,i}(t,\bY_{t})\cdot 
dB^j_{t}.
\end{split}
\end{equation}
Similarly, as $(u^{N,i})_{i \in \{1,\dots,N\}}$ satisfies \eqref{eq:uNi}, we have:
\begin{equation}
\label{rep:uNi}
\begin{split}
&d U^{N,i}_{t}
\\
&= 
\Bigl[ H \bigl( Y_{i,t},D_{x_i} u^{N,i}(t, \bY_t) \bigr)
- D_{x_{i}} u^{N,i}(t,\bY_{t}) 
\cdot
D_pH\bigl(Y_{i,t}, 
D_{x_i} u^{N,i}(t, \bY_t) \bigr)
\\
&\hspace{150pt}
- F^N \bigl( Y_{i,t},m^{N,i}_{\bY_t})
- r^{N,i}(t,\bY_{t})
\Bigr] {\rm dt}
\\
&\hspace{5pt}
-
\sum_{j }
D_{x_{j}} u^{N,i}(t,\bY_{t}) 
\cdot
\Bigl( 
D_pH \bigl(Y_{j,t}, D_{x_j}v^{N,j}(t, \bY_t) \bigr)
-
D_pH \bigl(Y_{j,t}, D_{x_j}u^{N,j}(t, \bY_t) \bigr)
\Bigr) {\rm dt}
\\
&\hspace{5pt} + 
\sqrt{2}
\sum_{j} D_{x_{j}} u^{N,i}(t,\bY_{t})
\cdot
dB^j_{t}.
\end{split}
\end{equation}
Computing the difference between 
\eqref{rep:vNi}
and \eqref{rep:uNi}, taking the square and applying It\^o's formula
again, we obtain:
\begin{equation}
\label{CvyxIneq2:ito}
\begin{split}
&d 
\bigl[ 
U^{N,i}_{t}
- 
V^{N,i}_{t}
\bigr]^2
\\
&= \biggl[ 2 \bigl( U^{N,i}_{t}
- 
V^{N,i}_{t}
\bigr) 
\cdot \Bigl(  
 H \bigl( Y_{i,t},DU^{N,i,i}_{t} \bigr)
 - H \bigl( Y_{i,t},DV^{N,i,i}_{t} \bigr)
 \Bigr)
\\ 
&\hspace{25pt} 
- 2 \bigl( U^{N,i}_{t}
- 
V^{N,i}_{t}
\bigr) 
\cdot \Bigl(  
DU^{N,i,i}_{t}
\cdot
\bigl[
 D_{p} H \bigl( Y_{i,t},DU^{N,i,i}_{t} \bigr)
 -
D_{p} H \bigl( Y_{i,t},DV^{N,i,i}_{t}
\bigr) \bigr] \Bigr)
\\ 
&\hspace{25pt} 
- 2 \bigl( U^{N,i}_{t}
- 
V^{N,i}_{t}
\bigr) 
\cdot \Bigl(  
\bigl[
DU^{N,i,i}_{t}
-
DV^{N,i,i}_{t}
\bigr]
\cdot
 D_{p} H \bigl( Y_{i,t},DV^{N,i,i}_{t} \bigr)
 \Bigr)
 \\
 &\hspace{25pt}
 - 2
 \bigl( U^{N,i}_{t}
- 
V^{N,i}_{t}
\bigr)
r^{N,i}(t,\bY_{t})
 \biggr] {\rm dt}
\\
&\hspace{10pt} 
- 2 \bigl( U^{N,i}_{t}
- 
V^{N,i}_{t}
\bigr) 
\sum_{j }
D U^{N,i,j}_{t} 
\cdot
\Bigl( 
D_pH \bigl(Y_{j,t}, DV^{N,j,j}_{t} \bigr)
-
D_pH \bigl(Y_{j,t}, DU^{N,j,j}_{t} \bigr)
\Bigr) {\rm dt}
 \\
 &\hspace{15pt}
 + 
 \biggl[ 2
 \sum_{j} \vert 
DU^{N,i,j}_{t} 
-
 DV^{N,i,j}_{t}
 \vert^2 
%  \\
% &\hspace{10pt}
+
\sqrt{2}
 \bigl( U^{N,i}_{t} - V^{N,i}_{t}
 \bigr) 
\sum_{j} 
\biggl[
\bigl(
DU^{N,i,j}_{t}
- DV^{N,i,j}_{t}
\bigr) 
\cdot
dB^j_{t}
\biggr].
\end{split}
\end{equation}
Recall now that 
$H$ 
and $D_{p} H$
are Lipschitz continuous in the variable $p$. Recall also the notation $\alpha^N$, $\beta^N$ and $r^N$ in \eqref{alphaNbetaN}. Integrating from $t$ to $T$ 
in the above formula and 
taking the conditional expectation 
given $\bZ$ (with the shorten notation 
$\E^{\bZ}[\cdot]
= \E[ \cdot \vert \bZ]$), we deduce: 
\begin{equation}
\label{eq:quelles:H:utiles:2}
\begin{split}
&{\mathbb E}^{\bZ}
\bigl[ 
\vert
U^{N,i}_{t}
- 
V^{N,i}_{t}
\vert^2
\bigr]
+ 
2  \sum_{j} \E^{\bZ}
 \biggl[ \int_{t}^T
 \vert 
DU^{N,i,j}_{s} 
-
 DV^{N,i,j}_{s}
 \vert^2 ds 
 \biggr]
 \\
&\hspace{15pt} \leq {\mathbb E}^{\bZ}
\bigl[ 
\vert
U^{N,i}_{T}
- 
V^{N,i}_{T}
\vert^2
\bigr]
+ r^N
\int_{t}^T 
\E^{\bZ} 
\bigl[
\vert
U^{N,i}_{s}
- 
V^{N,i}_{s}
\vert
\bigr]
ds
\\
&\hspace{30pt}
+ C(1+\alpha^N) 
\int_{t}^T 
\E^{\bZ}
\Bigl[
\vert
U^{N,i}_{s}
- 
V^{N,i}_{s}
\vert
\cdot
\vert 
DU^{N,i,i}_{s} 
-
DV^{N,i,i}_{s}
\vert 
\Bigr]
ds
\\
&\hspace{30pt} + 
C\beta^N
\sum_{j \not = i} 
\int_{t}^T 
\E^{\bZ}
\Bigl[
\vert
U^{N,i}_{s}
- 
V^{N,i}_{s}
\vert
\cdot
\vert 
DU^{N,j,j}_{s} 
-
DV^{N,j,j}_{s}
\vert 
\Bigr]
ds.
\end{split}
\end{equation}
Recall that $U^{N,i}_{T}=V^{N,i}_{T}=G(Y_{i,T})$.
By Young's inequality, we get
\begin{equation}
\label{CvyxIneq2:001-new-0}
\begin{split}
&{\mathbb E}^{\bZ}
\bigl[ 
\vert
U^{N,i}_{t}
- 
V^{N,i}_{t}
\vert^2
\bigr]
+  \E^{\bZ}
 \biggl[ \int_{t}^T
 \vert 
DU^{N,i,i}_{s} 
-
 DV^{N,i,i}_{s}
 \vert^2 ds 
  \biggr]
  \\
 &\leq C(r^N)^2
+
 C\left(1+(\alpha^N)^2+(N\beta^N)^2\right) 
\int_{t}^T 
\E^{\bZ}
\bigl[
\vert
U^{N,i}_{s}
- 
V^{N,i}_{s}
\vert^2
\bigr]
ds\\
& + 
\frac1{2N} \sum_{j}
\E^{\bZ}
 \biggl[ \int_{t}^T
 \vert 
DU^{N,j,j}_{s} 
-
 DV^{N,j,j}_{s}
 \vert^2 ds 
  \biggr].
\end{split}
\end{equation}
Summing over $i$ we obtain: 
\begin{equation}
\label{zvejhgbjkj}
\begin{split}
&\sum_i {\mathbb E}^{\bZ}
\bigl[ 
\vert
U^{N,i}_{t}
- 
V^{N,i}_{t}
\vert^2
\bigr]
+  \frac12 \sum_i \E^{\bZ}
 \biggl[ \int_{t}^T
 \vert 
DU^{N,i,i}_{s} 
-
 DV^{N,i,i}_{s}
 \vert^2 ds 
  \biggr]
  \\
 &\leq CN(r^N)^2
+
 C\left(1+(\alpha^N)^2+(N\beta^N)^2\right) 
\int_{t}^T 
\sum_i \E^{\bZ}
\bigl[
\vert
U^{N,i}_{s}
- 
V^{N,i}_{s}
\vert^2
\bigr]
ds.
\end{split}
\end{equation}
By Gronwall's Lemma, 
this leads to:
\begin{equation}
\label{CvyxIneq2:001-new}
\begin{split}
\sup_{t_0 \leq t \leq T}
\biggl[
\sum_i {\mathbb E}^{\bZ}
\bigl[ 
\vert
U^{N,i}_{t}
- 
V^{N,i}_{t}
\vert^2
\bigr] \biggr] \leq CN(r^N)^2\exp(C\theta^N), 
\end{split}
\end{equation}
where $\theta^N$ is given in \eqref{thetaN}. 
Plugging \eqref{CvyxIneq2:001-new} into 
\eqref{zvejhgbjkj}, we deduce that 
\begin{equation*}
 \sum_{j}
\E^{\bZ}
 \biggl[ \int_{t_0}^T
 \vert 
DU^{N,j,j}_{s} 
-
 DV^{N,j,j}_{s}
 \vert^2 ds 
  \biggr] \leq CN(r^N)^2\theta^N \exp(C\theta^N).
\end{equation*}
Inserting this bound in the right-hand side of \eqref{CvyxIneq2:001-new-0} and applying Gronwall's lemma once again, we finally end up 
with: 
\begin{equation}
\label{CvyxIneq2:001}
\begin{split}
&
\sup_{t \in [t_0,T]}
{\mathbb E}^{\bZ}
\bigl[ 
\vert
U^{N,i}_{t}
- 
V^{N,i}_{t}
\vert^2
\bigr]
+ \E^{\bZ}
 \biggl[ \int_{t_0}^T
 \vert 
DU^{N,i,i}_{s} 
-
 DV^{N,i,i}_{s}
 \vert^2 ds 
  \biggr]\\
& \qquad \leq 
 C(r^N)^2\theta^N\exp(C\theta^N).
\end{split}
\end{equation}
This gives \eqref{CvyxIneq2BIS}. By the definition of $U^{N,i}$ and $V^{N,i}$ this implies that 
$$
\left| u^{N,i}(t_0,\bZ)-v^{N,i}(t_0,\bZ)\right|\leq  Cr^N(\theta^N)^{1/2}\exp(C\theta^N)\qquad {\rm a.e.}.
$$
Then choosing $\bZ_i$ with a uniform law on $\T^d$ implies, by continuity of $u^{N,i}$ and $v^{N,i}$ that 
$$
\left| u^{N,i}(t_0,\bx)-v^{N,i}(t_0,\bx)\right|\leq  Cr^N(\theta^N)^{1/2}\exp(C\theta^N) \qquad \forall \bx\in (\T^d)^N,
$$
which shows \eqref{uNi-vNi}.

\bigskip

We now  estimate the difference $X_{i,t}-Y_{i,t}$, for $t \in [t_0,T]$
and $i \in \{1,\dots,N\}$. 
In view of the equation satisfied by the processes $(X_{i,t})_{t \in [t_0,T]}$ and by $(Y_{i,t})_{t \in [t_0,T]}$, we have
\begin{equation}
\label{eq:quelles:H:utiles:3}
\begin{split}
 |X_{i,t}-Y_{i,t}| &\leq   \int_{t_0}^t 
\bigl|D_pH \bigl(X_{i,s}, D_{x_{i}}
u^{N,i}(s,\bX_{s})
\bigr) -D_pH
\bigl(Y_{i,s}, D_{x_{i}}v^{N,i}(s,\bY_{s})
\bigr)
\bigr|
 ds 
\\
& \leq C \int_{t_0}^t 
(1+ \|D^2_{x_i,x_i}u^{N,i}\|_\infty )|X_{i,s}-Y_{i,s}|+ \sum_{j\neq i}  \|D^2_{x_i,x_j}u^{N,i}\|_\infty |X_{j,s}-Y_{j,s}| ds\\
& + \int_{t_0}^t  \bigl|D_pH \bigl(Y_{i,s}, D_{x_{i}}
u^{N,i}(s,\bY_{s})
\bigr) -D_pH
\bigl(Y_{i,s}, D_{x_{i}}v^{N,i}(s,\bY_{s})
\bigr)
\bigr|
 ds \\
&\leq  C\hat \alpha^N\int_{t_0}^t |X_{i,s}-Y_{i,s}|ds  + 
C\hat \beta^N \sum_{j\neq i} \int_{t_0}^t |X_{j,s}-Y_{j,s}|ds\\
& +C \int_{t_0}^T  
\bigl|
D U^{N,i,i}_{s}
-
D V^{N,i,i}_{s}
\bigr|  ds.
\end{split}
\end{equation}
Computing as before the inequality satisfied by the sum $\sum_i  |X_{i,t}-Y_{i,t}|$, using the exchangeability of the $((X^{N,i}, Y^{N,i})$ and  
\eqref{CvyxIneq2:001}, we obtain 
\eqref{estixi-yiBIS} thanks to Gronwall inequality.

\end{proof}

Following exactly the same argument as for Theorem 2.13 in \cite{CDLL}, we deduce: 

\begin{Corollary}
\label{cor:mainCVBIS}
Fix $N\geq 1$ and $(t_0,m_0)\in [0,T]\times \Pk$.  For any $i\in\{1,\dots, N\}$ and $x\in \T^d$, let us set 
$$
w^{N,i}(t_0,x_i,m_0) :=  \int_{(\T^d)^{(N-1)}} v^{N,i}(t_0, \bx) \prod_{j\neq i}m_0({\rm dx}_j) \qquad {\rm where }\; \bx=(x_1,\dots, x_N).
$$ 
Then 
$$
\left\| w^{N,i}(t_0,\cdot, m_0)-U^N(t_0,\cdot, m_0)\right\|_{\infty} \leq \left\{\begin{array}{ll}
 Cr^N(\theta^N)^{1/2} \exp(C\theta^N)+ C\|D_mU^N\|_\infty N^{-1/d} & {\rm if}\; d\geq 3,\\
 Cr^N(\theta^N)^{1/2} \exp(C\theta^N)+ C\|D_mU^N\|_\infty N^{-1/2}\log(N) & {\rm if}\; d=2,\\
  Cr^N(\theta^N)^{1/2} \exp(C\theta^N)+ C\|D_mU^N\|_\infty N^{-1/2} & {\rm if}\; d=1.
\end{array}\right. 
$$
where the constant $C$ does not depend on $t_0$, $m_0$, $i$ and $N$ and where  $\theta^N$ and $\hat \theta^N$ are defined before Theorem \ref{thm:CvyxBIS}.
\end{Corollary}

\begin{proof}
We use the the Lipschitz continuity of $U^N$ and a result by Fournier and Guillin \cite{fournier2015rate} to deduce that, for $d\geq 3$  and for any $x_i\in \T^d$, 
$$
\begin{array}{l}
 \ds  \int_{\T^{d(N-1)}} |u^{N,i}(t,\bx)- U^N(t,x_i, m_0)|\prod_{j\neq i} m_0({\rm dx}_j) \\  
\qquad \qquad \ds = \;  \ds \int_{\T^{d(N-1)}} |U^N(t,x_i,m^{N,i}_\bx)- U^N(t,x_i, m_0)|\prod_{j\neq i} m_0({\rm dx}_j)\\
\qquad \qquad \ds \leq  \|D_mU^N\|_\infty\int_{\T^{d(N-1)}} \dk(m^{N,i}_\bx,m_0)\prod_{j\neq i} m_0({\rm dx}_j) \; \leq \; C\|D_mU^N\|_\infty N^{-1/d}.
\end{array}
$$
If $d=1$ (respectively $d=2$), the $N^{-1/d}$ has to be replaced by $N^{-1/2}$ (respectively $N^{-1/2}\log(N)$). 
Combining Theorem \ref{thm:CvyxBIS} with the above inequality, we obtain therefore, for $d\geq 3$, 
$$
\begin{array}{l}
\ds \left| w^{N,i}(t_0,x_i, m_0)-U^N(t_0,x_i, m_0)\right| \\
\qquad  \ds = \left| \int_{\T^{d(N-1)}}v^{N,i}\bigl(t,(x_j)\bigr)\prod_{j\neq i}m_0({\rm dx}_j) -U^N(t,x_i, m_0)\right|\\
\qquad  \ds = \left\|v^{N,i}-u^{N,i}\right\|_\infty+ \int_{\T^{dN}} |u^{N,i}(t,\bx)- U^N(t,x_i, m_0)|\prod_{j\neq i} m_0({\rm dx}_j) \\
\qquad \ds \leq Cr^N(\theta^N)^{1/2} \exp(C\theta^N)+ C\|D_mU^N\|_\infty N^{-1/d}.
\end{array}
$$
As above, the last term is  $N^{-1/2}$ if $d=1$ and $N^{-1/2}\log(N)$ if $d=2$. 
\end{proof}

%%%%%%%%%%%%%%%%%%%%%%%%%%%%%%%%%%
%%%%%%%%%%%%%%%%%%%%%%%%%%%%%%%%%%
%%%%%%%%%%%%%%%%%%%%%%%%%%%%%%%%%%

\subsection{Putting the estimates together}

Here we fix a initial condition $(t_0,m_0)\in [0,T]\times \Pk$, where $m_0$ is a smooth, positive density. Let $v^{N,i}$ be the solution of the Nash system \eqref{Nash0}. Following the averaging procedure of Corollary \ref{cor:mainCVBIS}, we set 
$$
w^{N,i}(t_0,x,m_0) := \inte\dots \inte v^{N,i}(t_0, \bx) \prod_{j\neq i}m_0({\rm dx}_j) \qquad {\rm where }\; \bx=(x_1,\dots, x_N).
$$ 
Let $u$ be the solution to the MFG system \eqref{e.MFGsystLoc}. 

\begin{Theorem}\label{thm:main} We have 
\be\label{Cvw-u}
\left\| w^{N,i}(t_0,\cdot, m_0)-u(t_0,\cdot)\right\|_\infty \leq 
\left\{\begin{array}{ll}
 CN^{-\tfrac{1}{d}} K_{N,\alpha}^{15} \exp(CK_{N,\alpha}^{6})+C\left(k^{R,\alpha}_N\right)^{\tfrac{2}{d+2}} & {\rm if}\; d\geq 3\\
CN^{-\tfrac{1}{2}}\ln(N) K_{N,\alpha}^{15} \exp(CK_{N,\alpha}^6)+C\left(k^{R,\alpha}_N\right)^{\tfrac{1}{2}}  & {\rm if}\; d=2\\
CN^{-\tfrac{1}{2}} K_{N,\alpha}^{15} \exp(CK_{N,\alpha}^{6})+C\left(k^{R,\alpha}_N\right)^{\tfrac{2}{3}} & {\rm if}\; d=1.
\end{array}\right. 
\ee
where $R$ and $\alpha$ do not depend on  $N$ (but depends on $m_0$). 
\end{Theorem}

In particular, if $K_{N,\alpha}= O\left((\ln(N))^{\theta}\right)$ for some $\theta\in (0, 1/(6 d))$, then $w^{N,i}(t_0,\cdot, m_0)$ converges uniformly to $u(t_0,\cdot)$.  

\begin{proof} For $N\in \N$, let $u^N$ be the solution to the MFG system \eqref{e.MFG.t0m0}. 
Recalling Proposition \ref{prop.estiuep-u}, we have 
$$
\|u^N(t_0,\cdot)-u(t_0,\cdot)\|_\infty \leq C\left(k^{R,\alpha}_N\right)^{\tfrac{2}{d+2}},
$$
where $R$ is a bound on the $C^\alpha$ norm of the $m^N$ (Proposition \ref{prop.estiuep-u}). As 
$$
U^N(t_0,x,m_0)= u^N(t_0,x),
$$
Corollary \ref{cor:mainCVBIS} implies that (for $N\geq 3$), 
\begin{align*}
 \left\| w^{N,i}(t_0,\cdot, m_0)-u^N(t_0,\cdot)\right\|_\infty &\leq  Cr^N(\theta^N)^{1/2} \exp(C\theta^N)+ C\|D_mU^N\|_\infty N^{-1/d} \\
& \leq   
 CN^{-\tfrac{1}{d}} K_{N,\alpha}^{15} \exp(CK_{N,\alpha}^{6}). 
\end{align*}
\end{proof}

We now consider a particular case: 
\begin{Corollary}\label{prop.ex} Assume that 
$$
F^N(x,m)= F(\cdot, \xi^{\ep_N}\ast m(\cdot))\ast \xi^\ep(x),
$$
where $\xi^\ep$ are as in the example in Proposition \ref{prop.ex0}. If one chooses $\ep_N = \ln(N)^{-\beta}$, with $\beta \in (0, (6d(2d+15))^{-1})$, the convergence in \eqref{Cvw-u} is of order $A(\ln(N))^{-1/B}$ for some constants $A,B$. 
\end{Corollary}

%%%%%%%%%%%%%%%%%%%%%%%%%%%%%%%%%%%%%%%%%
%%%%%%%%%%%%%%%%%%%%%%%%%
\subsection{Convergence of the optimal solutions}

We complete the paper by a discussion on the convergence of the optimal solutions and a propagation of chaos. 

Let us explain the problem. Let $m_0\in \Pk$ with a smooth, positive density.  Let $(v^{N,i})$ be the solution to the Nash system \eqref{Nash0} and, for $t_0\in [0,T)$, $m_0\in \Pw$, $(u,m)$ be the solution to the MFG system \eqref{e.MFGsystLoc} starting at time $t_0$ from $m_0$. Let $(Z_i)$ be an i.i.d family of $N$ random variables of law $m_0$. We set $\bZ=(Z_1,\dots, Z_N)$.  Let also $((B^{i}_{t})_{t \in [ 0,T]})_{i \in \{1,\dots,N\}}$ be a family of $N$ independent Brownian motions which is also independent of $(Z_i)$. 
We consider the optimal trajectories $(\bY_t=(Y_{1,t},\dots, Y_{N,t}))_{t\in [t_0,T]}$ for the $N-$player game: 
$$
\left\{\begin{array}{l}
dY_{i,t}= -D_pH(Y_{i,t}, D_{x_i}v^{N,i}(t, \bY_t)){\rm dt} +\sqrt{2} dB^{i}_t, \qquad t\in [t_0,T],\\
Y_{i,t_0}= Z_i
\end{array}\right.
$$
and the optimal solution $(\tilde \bX_t=(\tilde X_{1,t}, \dots, \tilde X_{N,t}))_{t\in[t_0,T]}$ to the limit MFG system:
$$
\left\{\begin{array}{l}
d\tilde X_{i,t} =  -D_pH\left(\tilde X_{i,t}, D_xu
\bigl(t,\tilde X_{i,t}\bigr)\right){\rm dt} +\sqrt{2}dB^i_t,\qquad t\in [t_0,T], \\
\tilde X_{i,t_0}= Z_i. 
\end{array}\right.
$$
The next result provides an estimate of the distance between the solutions. To fix the ideas, we work in dimension $d\geq 3$. 

\begin{Theorem}\label{thm:CvMFG}  Under our standing assumptions, for any $N\geq 1$ and any $i\in \{1,\dots, N\}$, we have
$$
\E\biggl[ \sup_{t\in [t_0,T]} \left|Y_{i,t}-\tilde X_{i,t}\right| \biggr]\leq 
 C\left[ K_{N,\alpha}^{25/2} N^{-1/d} \exp\{CK_{N,\alpha}^6\}+\left(k^{R,\alpha}_N\right)^{\frac{2}{d+2}}\right],
$$
where the constant $C>0$ is independent of $t_0$ and $N$, but depends on $m_0$. 
\end{Theorem}

In particular, if $K_{N,\alpha}= O\left((\ln(N))^{\theta}\right)$ for some $\theta\in (0, 1/(6d))$, then the optimal trajectories $(Y_{i,t})$ converge to the 
$(\tilde X_{i,t})$ and become asymptotically i.i.d. 

In order to illustrate the result, let us come back to our main example: 
\begin{Proposition}\label{prop.choasex} Assume that $F^N$ are of the form \eqref{Fepex0}. Then, for $\ep_N= \ln(N)^{-\beta}$ for $\beta \in (0, [6d(2d+15)]^{-1})$, we have 
$$
\E\biggl[ \sup_{t\in [t_0,T]} \left|Y_{i,t}-\tilde X_{i,t}\right| \biggr]\leq A(\ln(N))^{-1/B},
$$
for some constants $A,B>0$. 
\end{Proposition}

\begin{proof}[Proof of Theorem \ref{thm:CvMFG}.]  Let $U^N$ be a solution to the master equation \eqref{MasterEq} and set $u^{N,i}(t,\bx)= U^N(t,x_i,m^{N,i}_{\bx})$. Let  $(X_{i,t})$ be the solution to
$$
\left\{\begin{array}{l}
dX_{i,t}= -D_pH\bigl(X_{i,t}, D_{x_i}u^{N,i}(t, \bX_t)\bigr){\rm dt} +\sqrt{2} dB^{i}_t\qquad t\in [t_0,T],\\
X_{i,t_0}= Z_i,
\end{array}\right.
$$
and $(\hat X_{i,t})$ be the solution to 
$$
\left\{\begin{array}{l}
d\hat X_{i,t} =  -D_pH\left(\hat X_{i,t}, D_xu^N
\bigl(t,\hat X_{i,t}\bigr)\right){\rm dt} +\sqrt{2}dB^i_t \qquad t\in [t_0,T],\\
\hat X_{i,t_0}= Z_i, 
\end{array}\right.
$$
where $(u^N,m^N)$ is the solution of the MFG system \eqref{e.MFG.t0m0}. Note that the $(\hat X_{i,t})$ are i.i.d. with law $(m^N(t))$. As, for any $t\in [t_0,T]$,  
$$
u^N(t,\cdot)= U^N(t, \cdot, m^N(t)), 
$$
the $(\hat X_{i,t})$ are also solution to 
$$
\left\{\begin{array}{l}
d\hat X_{i,t} =  -D_pH\left(\hat X_{i,t}, D_xU^N
\bigl(t,\hat X_{i,t}, m^N(t)\bigr)\right){\rm dt} +\sqrt{2}dB^i_t \qquad t\in [t_0,T],\\
\hat X_{i,t_0}= Z_i. 
\end{array}\right.
$$
 The main step of the proof is the following claim:
\be
\label{xit-tildexit}
\E \Bigl[ \sup_{t\in [t_0,T] } \bigl| X_{i,t}-\hat X_{i,t}\bigr| \Bigr] \leq CK_{N,\alpha}  N^{-1/d} \exp\{CK_{N,\alpha}\}.
\ee
Let us fix $ i\in \{1, \dots, N\}$ and let 
$$
\rho(t)= \E \Bigl[ \sup_{s\in [t_0,t] }\bigl| X_{i,s}-\hat X_{i,s}\bigr|\Bigr]. 
$$
Then, for any $t_0\leq s\leq t\leq T$, we have  
\begin{align*}
\bigl| X_{i,s}-\hat X_{i,s}\bigr|
&\leq \int_{t_0}^s 
\bigl|  -D_pH\bigl(X_{i,r}, D_{x}U^N(r, X_{i,r}, m^{N,i}_{\bX_r})\bigr)+D_pH \bigl(\hat  X_{i,r}, D_xU^N \bigl(r, \hat X_{i,r},
m^N(r)\bigr)
\bigr)  \bigr| {\rm dr} 
\\
& \leq C\int_{t_0}^s \bigl| X_{i,r}-\hat X_{i,r}\bigr|+ \bigl|D_xU^N(r, X_{i,r}, m^{N,i}_{\bX_r})-D_xU^N \bigl(r, \hat X_{i,r},m^N(r)\bigr)\bigr| {\rm dr},
\end{align*}
where we have used the fact that $D_pH$ is globally Lipschitz continuous. As the map $(x,m)\to D_xU^N(r,x,m)$ is Lipschitz continuous with constant $CK_{N,\alpha}$ (Theorem \ref{them:KeyEsti}),  we get
\begin{align*}
\bigl| X_{i,s}-\hat X_{i,s}\bigr|
\leq CK_{N,\alpha} \int_{t_0}^s 
\Bigl(|X_{i,r}-\hat  X_{i,r}|+ \dk\bigl(m^{N,i}_{\bX_r}, m^{N,i}_{\hat \bX_r}\bigr) 
+ \dk \bigl( m^{N,i}_{\hat \bX_r}, m^N(r)\bigr)
\Bigr) {\rm dr}, 
\end{align*}
where 
\be\label{dkmNXtildeX}
\dk\bigl(m^{N,i}_{\bX_r}, m^{N,i}_{\hat \bX_r}\bigr)
\leq \frac{1}{N-1} \sum_{j\neq i} |X_{j,r}-\hat X_{j,r}|.
\ee
Hence
\begin{equation}\label{oialenrdf}
\bigl| X_{i,s}-\hat X_{i,s}\bigr| \leq 
CK_{N,\alpha}\int_{t_0}^s 
\Bigl(
 |X_{i,r}-\hat  X_{i,r}|+ \frac{1}{N-1} \sum_{j\neq i} |X_{j,r}-\hat X_{j,r}|+ \dk ( m^{N,i}_{\hat \bX_r}, m^N(r)) \Big) dr. 
\end{equation}
As the $(\hat X_{i,t})$ are i.i.d. and $d\geq 3$, we have from \cite{fournier2015rate} that   
$$
\E\Bigl[\dk\Bigl(m^{N,i}_{\hat \bX_r}, m^N(r)\Bigr)\Bigr]
\leq C N^{-1/d}.
$$
So, taking the supremum over $s\in [t_0,t]$ in \eqref{oialenrdf} and then the expectation, gives, since the random variables $(X_{j,r}-\hat X_{j,r})_{j \in \{1,\dots,N\}}$ have the same law: 
$$
\rho(t) =\E\Bigl[ \sup_{s\in [t_0,t]} \bigl| X_{i,s}-\hat X_{i,s}\bigr|\Bigr] 
\leq  CK_{N,\alpha}\int_{t_0}^t \rho(s) ds + CK_{N,\alpha} N^{-1/d}.
$$
 Then Gronwall inequality gives \eqref{xit-tildexit}. \\

We now complete the proof by recalling that, from Theorem \ref{thm:CvyxBIS}, 
$$
\E\bigl[\sup_{t\in [t_0,T]} |Y_{i,t}-X_{i,t}|\bigr]\leq Cr^N(\theta^N\hat \theta^N)^{1/2}\exp(C(\theta^N+\hat \theta^N)),
$$
where 
$$
r^N\leq \frac{CK_{N,\alpha}^{12}}{N}, \qquad \theta^N  \leq CK_{N,\alpha}^{6}, \qquad  \hat \theta^N \leq CK_{N,\alpha}^{3}.
$$
On the other hand, Corollary \ref{cor:solEDS} states that 
$$
\E\left[\sup_{t\in  [t_0,T]} \left|\widetilde X_t-\widehat X_t\right|\right]\leq C\left(k^{R,\alpha}_N\right)^{\frac{2}{d+2}},
$$
Therefore 
$$
\E\biggl[ \sup_{t\in [t_0,T]} \left|Y_{i,t}-\tilde X_{i,t}\right| \biggr]\leq C\left[ K_{N,\alpha} N^{-1/d} \exp\{CK_{N,\alpha}\}+ K_{N,\alpha}^{33/2}N^{-1}\exp(CK_{N,\alpha}^6)+\left(k^{R,\alpha}_N\right)^{\frac{2}{d+2}}\right].
$$

\end{proof}

%%%%%%%%%%%%%%%%%%%%%%%%%%%%%%%%%%%%%%%%%%%%%%%%%%%%%%%%%%%%%%%%%%%%%%%%%%%%%
\section{Appendix}

In the appendix, we state an estimate for equations of the form: 
\be\label{eq:appen}
\left\{\begin{array}{l}
\ds \partial_t w - \Delta w + V(t,x)\cdot Dw =f \qquad {\rm in}\; [0,T]\times \T^d,\\
\ds w(0,x)=w_0(x)\qquad {\rm in}\; \T^d,
\end{array}\right.
\ee
where $V$ is a fixed bounded vector field. 

\begin{Proposition}\label{prop.reguregu} If $w$ is a solution to the above equation with $w_0\in C^{1+\alpha}$, then 
$$
\sup_{t\in [0,T]}\|w(t)\|_{1+\alpha}\leq C\left[\|w_0\|_{1+\alpha}+\|f\|_\infty\right],
$$
where $C$ depends on $\|V\|_\infty$, $T$, $\alpha$ and $d$ only. 
\end{Proposition}

This kind of estimate is standard in the literature: for instance Theorem IV.9.1 of \cite{LSU} (and its Corollary) states that $Dw$ is bounded in $C^{\beta/2,\beta}$ for any $\beta\in (0,1)$. However the bound might depend on the vector field $V$ and not only on its norm. We only check this is not the case. 

\begin{proof} Let us first check that the result holds for an homogenous initial datum. More precisely, we prove in a first step that, if $w$ solve \eqref{eq:appen} with $w(0,\cdot)=0$, then $\|Dw\|_\infty\leq C \|f\|_\infty$, where the constant $C$ depends on $\|V\|_\infty$, $T$ and $d$ only. For this we argue by contradiction and assume for a while that there exists $V_n$ and $f_n$, bounded in $L^\infty$, and $w_n$ such that 
$$
\left\{\begin{array}{l}
\ds \partial_t w_n - \Delta w_n + V_n\cdot Dw_n =f_n \qquad {\rm in}\; [0,T]\times \T^d\\
\ds w_n(0,x)=0\qquad {\rm in}\; \T^d.
\end{array}\right.
$$
with $k_n:=\|Dw_n\|_\infty\to +\infty$. We set $\tilde w_n:= w_n/k_n$, $\tilde f_n:= f_n/k_n$. Then $\tilde w_n$ solves the heat equation with a right-hand side $\tilde f_n-V_n\cdot D\tilde w_n$ which is bounded in $L^\infty$. By standard estimates on the heat potential (see (3.2) of Chapter 3 in \cite{LSU}), $\partial_t \tilde w_n$ and $D^2\tilde w_n$ are bounded in $L^p$ for any $p$ independently of $n$. Then a Sobolev type inequality (Lemma II.3.3 in \cite{LSU}) implies that $D\tilde w_n$ is bounded in $C^{\beta/2,\beta}$ independently of $n$ for any $\beta\in (0,1)$. On the other hand, $(\tilde f_n)$ tends to $0$ in $L^2$ and, by standard energy estimates, $(D\tilde w_n)$ tends to $0$ in $L^2$. This is in contradiction with the fact that $\|D\tilde w_n\|_\infty=1$ and that $D\tilde w_n$ is bounded in $C^{\beta/2,\beta}$. So we have proved that there exists a constant $ C$, depending on $\|V\|_\infty$, $d$ and $T$ only, such that the solution to \eqref{eq:appen} with $w(0,\cdot)=0$ satisfies $\|Dw\|_\infty\leq  C\|f\|_\infty$. Using the same argument on the the heat potential as above yields to 
$$
\|Dw\|_{C^{\beta/2,\beta}}\leq C_\beta \|f-V\cdot Dw\|_\infty\leq C_\beta  \|f\|_\infty, 
$$
where $C_\beta$ depends on $\|V\|$, $d$, $T$ and $\beta$ only. 

We now remove the assumption that $w_0=0$. We rewrite $w$ as the sum $w=w_1+w_2$ where $w_1$ solves the heat equation with initial condition $w_0$ and $w_2$ solves equation \eqref{eq:appen} with  right-hand side $f-V\cdot Dw_1$ and initial condition $w_2(0,\cdot)=0$. By maximum principle, we have 
$$
\sup_{t\in [0,T]} \|Dw_1(t)\|_{\alpha}\leq C\|Dw_0\|_{\alpha}.
$$
By the first step of the proof, we also have, for any $\beta\in (0,1)$,  
$$
\|Dw_2\|_{C^{\beta/2,\beta}}\leq C_\beta\|f-V\cdot Dw_1\|_\infty \leq C_\beta(\|f\|_\infty+ \|Dw_0\|_{\alpha}). 
$$
Choosing $\beta=\alpha$ then gives the result. 
\end{proof}

%%%%%%%%%%%%%%%%%%%%%%%%%%%%%%%%%%%%%%%%%%%%%%%%%%%%%%%%%%%%%%%%%%%%%%%%%%%%%%
%%%%%%%%%%%%%%%%%%%%%%%%%%%%%%%%%%%%%%%%%%%%%%%%%%%%%%%%%%%%%%%%%%%%%%%%%%%%%%
%%%%%%%%%%%%%%%%%%%%%%%%%%%%%%%%%%%%%%%%%%%%%%%%%%%%%%%%%%%%%%%%%%%%%%%%%%%%%%

\bibliographystyle{plain}
%\bibliography{Biblio}

\end{document}